\documentclass[12pt]{amsart}

\usepackage{graphicx, xcolor}
\usepackage{comment}
\usepackage{appendix}

\usepackage{latexsym,amsmath,amsthm,amsfonts,mathrsfs,MnSymbol,epsfig}
\usepackage[initials]{amsrefs}
\usepackage[all]{xy}
\usepackage{bm}
\usepackage{slashed}
\usepackage{amsthm}
\usepackage{thmtools}

\usepackage{stackengine}
\stackMath

\providecommand{\CC}{{\mathbb{C}}}
\providecommand{\RR}{{\mathbb{R}}}

\providecommand{\ZZ}{{\mathbb{Z}}}

\providecommand{\KK}{{\mathcal K}}
\providecommand{\UU}{{\mathcal U}}

\providecommand{\Wp}{{\mathcal W}^{pol}}

\providecommand{\mcB}{{\mathcal{B}}}

\providecommand{\mcE}{{\mathcal{E}}}

\providecommand{\mcK}{{\mathcal{K}}}
\providecommand{\mcL}{{\mathcal{L}}}

\providecommand{\mcS}{{\mathcal{S}}}
\providecommand{\mcW}{{\mathcal{W}}}

\providecommand{\mfA}{{\mathfrak{A}}}

\providecommand{\Lh}{{\mathfrak h}}

\providecommand{\Sch}{{\mathcal{S}}}

\providecommand{\Heis}{\mathbb{H}}

\providecommand{\SymH}{\Sigma_H}

\newcommand{\ang}[1]{\langle #1 \rangle} 
\newcommand{\lra}{\longrightarrow}

\DeclareMathOperator{\Hom}{Hom}
\DeclareMathOperator{\End}{End}

\DeclareMathOperator{\ch}{ch}
\DeclareMathOperator{\Td}{Td}

\DeclareMathOperator{\tr}{tr}

\DeclareMathOperator{\Ind}{Index}

\newtheorem{theorem}{Theorem}
\newtheorem{lemma}[theorem]{Lemma}

\newtheorem{proposition}[theorem]{Proposition}

\theoremstyle{definition}
\newtheorem{definition}[theorem]{Definition}

\theoremstyle{remark}
\newtheorem{remark}[theorem]{Remark}
\newtheorem{example}[theorem]{Example}

\numberwithin{equation}{section}

\title
[Toeplitz Operators and Equivariant $K$-homology]
{Toeplitz Operators  on Contact Manifolds and Equivariant $K$-homology}
\author{Alexander Gorokhovsky}
\address{University of Colorado, Boulder, Campus Box 395, Boulder, Colorado, 80309, USA}
\email{gorokhov@colorado.edu}
\author{Erik van Erp}
\address{Dartmouth College, 6188, Kemeny Hall, Hanover, New Hampshire, 03755, USA}
\email{erikvanerp@dartmouth.edu}
\subjclass[2010]{46L80, 58J20, 58J40}

\begin{document}

\begin{abstract}
We present an equivariant generalization of Boutet de Monvel's index theorem for Toeplitz operators on  contact manifolds.
We prove that the Dirac operator and the Szeg\"o projection determine the same class in equivariant $K$-homology, generalizing a theorem of Baum-Douglas-Taylor.
We do not assume that the contact manifold is the boundary of a strictly pseudoconvex  domain.
The proof proceeds by a deformation linking the principal symbols of the classical and Heisenberg pseudodifferential calculi.
At the level of symbols, the projection defining the Dirac class deforms to the principal Heisenberg symbol of the Szeg\"o projection.
This  deformation implies  equality of the corresponding classes in $K$-homology. This, in turn, gives an equivariant generalization of Boutet de Monvel's index formula for  Toeplitz operators.
\end{abstract}
\maketitle

\section{Introduction}
Let $M$ be a compact contact manifold  on which a compact Lie group $G$ acts smoothly, preserving the contact form.
In this paper we prove that the Spin$^c$ Dirac operator and the Szeg\"o projection on $M$ determine the same class in equivariant $K$-homology. 

Let $\Omega$ be a compact complex analytic domain with a smooth strictly pseudoconvex  boundary $\partial\Omega=M$.
The Szeg\"o projection $S$ is the orthogonal projection of $L^2(M)$ onto the Hardy space $H^2(M)$ of boundary values of holomorphic functions on $\Omega$.
If $f:M\to \mathrm{GL}(r,\CC)$ is a smooth function, let $T_f=SfS$ be the corresponding Toeplitz operator acting on $H^2(M,\CC^r)$. 
Boutet de Monvel proved the index  formula \cite{Bo79},
\begin{equation}\label{eq:BdM}
     \Ind T_f = \int_M \ch(f)\wedge \Td(T^{1,0}\Omega|M)
\end{equation}
In \cite{BDT89}, Baum, Douglas and Taylor reformulated Boutet de Monvel's theorem in $K$-homology and proved that the Szeg\"o projection and the Spin$^c$ Dirac operator determine the same class in $K$-homology,
\[
[D]=[S]\quad \text{in } K_1(M)
\]
The aim of this paper is to extend this in two directions:
\begin{itemize}
    \item We prove that $[D]=[S]$ when $M$ is a compact contact manifold, not necessarily the boundary of a  complex domain.
    \item We prove this equality in equivariant $K$-homology, if a compact Lie group $G$ acts on $M$ preserving the contact form.
\end{itemize}
The proof that $[D]=[S]$  in \cite{BDT89} relies on an analysis of the boundary map in $K$-homology $\partial:K_0(\Omega, M)\to K_1(M)$. 
This proof does not apply when the contact manifold $M$ is not fillable.
Moreover, even if $M$ is fillable, an action of a compact Lie group $G$ on a boundary $M=\partial \Omega$ does not necessarily extend to $\Omega$. Thus, the proof in \cite{BDT89} does not generalize to the equivariant case.

Our main result is the following theorem.

\begin{theorem}\label{thm:main}
Let $M$ be a compact contact manifold, and let $G$ be a compact Lie group that acts smoothly on $M$,
preserving the contact form.
Let $D$ be a $G$-equivariant Spin$^c$ Dirac operator on $M$, and let $S$ be a $G$-equivariant Szeg\"o projection on $M$.
Then in $G$-equivariant $K$-homology,
\[ [D]=[S]\qquad \text{in } K_1^G(M)\]
\end{theorem}
The Dirac operator is elliptic and defined in the classical symbol calculus, whereas the Szeg\"o projection is a pseudodifferential operator in the Heisenberg calculus adapted to the contact structure.
Our proof proceeds by a deformation that links these two calculi.
At the level of symbols, the projection defining the Dirac class deforms to the principal Heisenberg symbol of the Szegő projection.
Fiberwise, this deformation reduces to the deformation of the Bott generator to a rank one projection, as in the proof of Bott periodicity by Elliott, Natsume and Nest \cite{ENN93}.
A homotopy of elliptic symbols within the classical  calculus lifts to a homotopy of Fredholm operators. However, the symbols of the Dirac operator and the Szeg\"o projection lie in different calculi. Nevertheless, we show that this  deformation  implies an equivalence in $K$-homology of the corresponding  operators.

Theorem \ref{thm:main}  implies an equivariant version of Boutet de Monvel's  formula (\ref{eq:BdM}).

\begin{restatable}{theorem}{BDM}\label{cor:equivariant BdM}
Let  $\pi:G\to U(r)$ be a unitary representation of $G$, 
and let $f:M\to \mathrm{GL}(r,\CC)$ be a $G$-equivariant smooth function, i.e. $f(g.p)=\pi(g)f(p)\pi(g)^{-1}$.
Let $T_f$ be the corresponding Toeplitz operator.
Then the  equivariant index of $T_f$ is
\begin{equation}\label{eq:BdM equivariant intro}
    \chi_g(\Ind_G T_f) = \int_{M^g} \ch_g(f)\wedge \Td(M^g)\wedge \frac{1}{\ch_g(\lambda_{-1}(N^g)^*)}   
\end{equation}   
\end{restatable}
Here:
\begin{itemize}
\item $\chi_g(\pi)=\mathrm{Tr}(\pi(g))$ is the character of a representation $\pi\in R(G)$ evaluated at $g\in G$.
\item $\ch_g$ is the equivariant Chern character.
\item $M^g=\{p\in M\mid gp=p\}$ is the fixed point set of $g$. The contact form $\theta$ of $M$ restricts to a contact form on $M^g$ (Lemma \ref{lem:restriction of theta to M^g}). 
\item $\Td(M^g)$ is the Todd class of the contact manifold $M^g$, viewed as a stably almost complex manifold. 
\item $N^g$ is the normal bundle of $M^g$ in $M$, viewed as a  complex vector bundle. 
\item $\lambda_{-1}(N^g)^*$ is the formal alternating sum $\sum (-1)^k \Lambda^k (N^g)^* $.
\end{itemize}

See section \ref{sec:equivariant BdM formula} for details.
\vskip 6pt

Taking $g=e$  to be  the identity element of $G$,
$\chi_e(\Ind_G T_f)$ is the ordinary $\ZZ$-valued index of $T_f$.
Thus (\ref{eq:BdM equivariant intro}) includes Boutet de Monvel's formula  (\ref{eq:BdM}) as a special case.
To our knowledge, the equivariant generalization of this theorem is new.

In \cite{GvE22} we presented a local index formula for Heisenberg elliptic operators on contact manifolds.
We proved our general formula by reduction to Boutet de Monvel's theorem. 
The present paper is  a  step toward an equivariant version of the formula of \cite{GvE22}.

\vskip 6pt
The paper is organized as follows.
Section \ref{sec:Heisenberg and Weyl} reviews basic facts about the Heisenberg group and the Weyl algebra, fixing the analytic framework used throughout.
Section \ref{sec:Bott per and quantization} interprets  equivariant Bott periodicity in terms of deformation quantization in the Weyl algebra.
Section \ref{sec:symbols in K theory} uses this deformation of the Bott generator to construct a deformation from the classical Dirac symbol to the Heisenberg symbol of the Szeg\"o projection. Finally, section \ref{sec:symbols to operators} completes the proof of Theorem \ref{thm:main} and Theorem \ref{cor:equivariant BdM}.

\section{The Heisenberg Group and the Weyl algebra}\label{sec:Heisenberg and Weyl}

In this section we review basic facts about the Heisenberg group and the Weyl algebra,
primarily to fix notation and conventions, especially regarding constants and signs.
A standard reference for this material is \cite{Fo89}.

\subsection{Symplectic vector spaces}\label{sec:symplectic vector spaces}

Throughout this section $V$ is a $2n$-dimensional symplectic vector space with symplectic form $\omega$.
We shall assume that  a complex structure $J$ has been chosen that is compatible with the symplectic form $\omega$:
\[
J^2=-I\qquad \omega(Jv, Jw) =\omega(v, w)\qquad  \omega(v, Jv) > 0 \text{ for } v\ne 0
\]
Then $g(v,w)=\omega(v,Jw)$ is a positive definite inner product on $V$,
and $J$ is orthogonal for $g$, i.e.  $g(Jv,Jw)=g(v,w)$.
$V$ has the structure of a complex vector space, where $i$ acts as $J$.
We denote this $n$-dimensional complex vector space by $V^{1,0}$.

There is a Hermitian form on $V^{1,0}$, complex linear in the first variable, defined by
\[ h(v, w):= g(v, w)-i \omega(v, w)\]
We sometimes write $\ang{v,w}=h(v,w)$ and $v\cdot w=g(v,w)$ to denote the hermitian and Euclidean inner products.

Let the unitary group
\[U:=U(V^{1,0})\cong U(n)\]
act in the standard way on $V$, $V^{1,0}$, $\Lambda^\bullet V^{1,0}$, etc.

The inner product $g$ determines an isomorphism $V \cong V^*$ by $v \mapsto g(v, \cdot)$. 
This isomorphism  maps $\omega$ to $\omega^*$, $J$ to $J^*$, and $g$ to $g^*$.

In various integrals below, $dv$, $dx$, $dy$ etc.\ denote  Lebesgue measure on $V$,
\[ dv = \frac{1}{n!}|\omega^n|\]
and similarly on $V^*$ using $\omega^*$.

\subsection{Convolution on the Heisenberg group}

The Heisenberg Lie algebra $\Lh=V\oplus \RR$ has bracket
\[ [(v,t),(w,s)]=(0,-\omega(v,w))\in \RR\qquad v,w\in V,\,t,s\in \RR\]
The Heisenberg group $\Heis=V\times\RR$ has group operation
\begin{align*}
    (v,t)\cdot (w,s) &= (v+w, t+s-\frac{1}{2}\omega(v,w))\\
    &= (v+w, t+s+\frac{1}{2}\mathrm{Im}\ang{v,w})
\end{align*} 

\begin{definition}\label{def:moyal product}
Let  $g,h\in \mcS(V)$ be two Schwartz class functions on the symplectic vector space $V$.
For each $\lambda\in \RR$, the twisted convolution product $\ast_\lambda$ is,
 \[ (g\ast_\lambda h)(v) = \int e^{\frac{1}{2}i\lambda\omega(v,w)} g(v-w)h(w)dw\]
Let  $g,h\in \mcS(V^*)$ be two Schwartz class functions on the dual space $V^*$.
For nonzero $\lambda\in\RR\setminus \{0\}$ the Moyal product $\#_\lambda$ is
\[ (g\#_\lambda h)(x) = \frac{1}{(\lambda\pi)^{2n}}\int e^{\frac{2}{\lambda}i \omega^*(a,b)} g(x+a)h(x+b)\,da\,db\]
For $\lambda=0$ we let $\#_0$ be pointwise multiplication of functions on $V^*$.
We sometimes write $\#$ for the Moyal product $\#_1$ with $\lambda=1$.
\end{definition}
\begin{remark}
If $\lambda\ne 0$ the definition of $\#_\lambda$ 
uses the dual symplectic form $\omega^*$ on $V^*$.
Note that $(\lambda\omega)^*=\lambda^{-1}\omega^*$.
If we replace $da\,db$ by the measures on $V^*$ determined by $\lambda^{-1}\omega^*$, we see that $\#_\lambda$ is simply the Moyal product $\#$ for the symplectic space $(V^*,\lambda^{-1}\omega^*)$.
\end{remark}

We record two standard identities (proofs are routine).

\begin{lemma}
For  Schwartz class functions $g,h$ on the Heisenberg group $\Heis=V\times \RR$ with convolution product $f=g\ast h$
we have for each $\lambda\in\RR$,
\[ f_\lambda = g_\lambda\ast_\lambda h_\lambda\]
where $f_\lambda, g_\lambda, h_\lambda$ are functions on $V$ obtained by Fourier transform in the $t\in \RR$ variable,
\[ f_\lambda(v) :=  \int e^{-i\lambda t}f(v,t)dt\]
\end{lemma}

\begin{lemma}
Fourier transform 
\[\hat{f}(x) = \int e^{-ixv} f(v) dv\]
is an algebra isomorphism
\[ \mcS(V,\ast_\lambda)\cong \mcS(V^*,\#_\lambda) \qquad f\mapsto \hat f\]
i.e., for  Schwartz class functions $g,h\in \mcS(V)$ with twisted convolution product $f=g\ast_\lambda  h$
we have $\hat{f}=\hat{g}\#_\lambda  \hat{h}$.
\end{lemma}

If $\lambda=0$ then $\ast_0$ is the standard convolution product of functions on $V$.
Under Fourier transform, convolution becomes pointwise multiplication.

For $\lambda>0$ we have algebra isomorphisms,
\[ \phi_\lambda:\mcS(V,\ast_1)\cong \mcS(V,\ast_\lambda)\quad \phi_\lambda(f)(v) = \lambda^{n}f(\sqrt{\lambda} v)\]

\begin{example}\label{composition example}
Let $\alpha, \beta \ge 0$. Then for $x\in V^*$,
\[
e^{-\alpha\|x\|^2} \#_\lambda e^{-\beta \|x\|^2} = \frac{1} {(1+\lambda^2 \alpha \beta)^n} 
e^{-\frac{\alpha +\beta}{1+\lambda^2\alpha \beta} \|x\|^2}
\]
With $\lambda> 0$ and $\alpha=\beta=1/\lambda$ we get 
\[e^{-\frac{1}{\lambda}\|x\|^2}\#_\lambda e^{-\frac{1}{\lambda}\|x\|^2} = 2^{-n} e^{-\frac{1}{\lambda}\|x\|^2}\]
Multiply both sides by $2^{2n}$ to get an idempotent in $\mcS(V^*,\#_\lambda)$,
\[
s_\lambda \#_\lambda s_\lambda=s_\lambda \qquad s_\lambda(x)=2^ne^{-\frac{1}{\lambda}\|x\|^2} 
\]

\end{example}

\subsection{The Weyl  algebra}\label{sec:Weylalg}

The Weyl  algebra $\mcW_\lambda=\mcW_\lambda(V,\omega)$ is the $\ZZ$-filtered algebra consisting of smooth functions  $f:V^*\to \CC$ that have an asymptotic expansion at infinity modulo Schwartz class functions,
\begin{equation*} f \sim \sum_{j=-m}^\infty f_j  \qquad f_j(s x) = s^{-j}f_j(x)\quad x\ne 0,\;s>0\end{equation*}
where each $f_j$ is homogeneous of degree $-j$.
For Schwartz class functions $a, b\in \mcS(V^*)$, the product $a\#_\lambda b$ in the Weyl algebra is as in Definition \ref{def:moyal product}.
This  product extends by continuity to $\mcW_\lambda$. 
The constant function $1$ is the unit in $\mcW_\lambda$.
$\mcS(V^*)$ is a two-sided ideal in $\mcW_\lambda$, which we denote by $\mcS_\lambda$.

If $\lambda=0$, the product of elements in  $\mcW_0$ is the usual commutative product of smooth functions on $V^*$. 

We let $\#=\#_1$ and $\mcW=\mcW_1$.
Note that for $\lambda\ne 0$,
\[ \mcW_\lambda(V,\omega) = \mcW(V,\lambda\omega)\]
Let $\Wp_\lambda$ be the polynomial subalgebra of $\mcW_\lambda$.
If we regard  $v\in V$ as a linear function $v:V^*\to \RR$ (by evaluation), then 
\[ v\#_\lambda w = vw + \frac{i\lambda}{2}\omega(v,w)\qquad v,w\in V\]
with commutator
\[ [v,w]_{\#_\lambda} = i\lambda\omega(v,w)\]
Thus, $\Wp_\lambda$ is isomorphic to the quotient algebra
\[ \Wp_\lambda\cong (\bigotimes V \otimes_\RR \CC)/I_\lambda\]
where  $I_\lambda$ is the ideal in the complexified tensor algebra,
\[ I_\lambda = \ang{v\otimes w-w\otimes v - i\lambda\omega(v,w)\mid v,w\in V}\subset \bigotimes V\otimes_\RR\CC\]
As usual, we identify polynomials  on $V^*$ with symmetric tensors in $\mathrm{Sym} V$.
A product $v_1v_2\cdots v_p$ of linear functions $v_j:V^*\to\RR$ is identified with the symmetric tensor
\[\frac{1}{p!}\sum_{\sigma\in S_p} v_{\sigma(1)}\otimes v_{\sigma(2)}\otimes \cdots\otimes v_{\sigma(p)}\]
If $f,g$ are two polynomials on $V^*$ represented by symmetric tensors in $\mathrm{Sym} V\otimes \CC$, then  modulo $I_\lambda$ the tensor product $f\otimes g$ is equal to $f\#_\lambda g$.

\begin{remark}
Fourier transform in the $t$ variable maps an element  $(v,t) \in \Lh=V\oplus \RR$ in the Lie algebra
to the linear function $iv+i\lambda:V^*\to \CC$.
The Lie bracket $[(v,0),(w,0)]=(0,-\omega(v,w))$
agrees with the commutator  $[iv,iw]_{\#_\lambda}=-i\lambda\omega(v,w)$ in the Weyl algebra. 
Thus, Fourier transform in $t$ gives an algebra homomorphism for each $\lambda\in\RR$,
\[\UU(\Lh)\otimes_\RR \CC \to \Wp_\lambda \]
\end{remark}

\subsection{Weyl pseudodifferential calculus}

Let  $e_1,\dots,e_n$ be an  orthonormal basis of the hermitian vector space $V^{1,0}$.
If we let $f_j=Je_j$, then $e_1,\dots, e_n, f_1, \dots , f_n$ is a symplectic basis of $(V,\omega)$,
\[ \omega(e_j,e_k)=\omega(f_j,f_k) = 0\qquad \omega(e_j,f_k)=\delta_{jk}\]
For $\lambda\ne 0$, represent elements of $V$ by unbounded  self-adjoint operators on $L^2(\RR^n)$, 
\[ e_j\mapsto x_j \qquad f_j\mapsto -i\lambda\frac{\partial}{\partial x_j} \]
Then 
the tensor algebra $\bigotimes V\otimes_\RR\CC$ is represented by differential operators on $\RR^n$ with polynomial coefficients.
This representation annihilates the ideal $I_\lambda$ and we obtain a representation of the polynomial Weyl algebra $\Wp_\lambda$.

More generally, for $\lambda\ne 0$ a function $f\in \mcW_\lambda$ is represented by the pseudodifferential operator 
\[\pi_\lambda(f):C_c^\infty(\RR^n)\to C_c^\infty(\RR^n)\]
\[ (\pi_\lambda(f)u)(x)= \frac{1}{(2\pi)^n}\int e^{i(x-y)\cdot \xi} f\left(\frac{x+y}{2},\lambda\xi\right)u(y)\,d\xi\,dy \]
Here we  write  $f:V^*\to\CC$ as a function $f(x,\xi)$ of $(x,\xi)\in \RR^{2n}$, where we identify $V^*=\RR^{2n}$ using the dual basis, i.e. 
\[ V^*\mapsto \RR^{2n}\qquad z\mapsto (x,\xi)\qquad e_j(z)=x_j, f_j(z)=\xi_j\]
Then 
\[  \pi_\lambda(e_j)= x_j \qquad \pi_\lambda(f_j)= -i\lambda\frac{\partial}{\partial x_j}\]
and
\[\pi_\lambda(f\#_\lambda g)=\pi_\lambda(f)\pi_\lambda(g)\]
If $f\in \mcW^0_\lambda(V)$ is of order zero, then $\pi_\lambda(f)$ extends to a bounded linear operator on $L^2(\RR^n)$.

For an order zero element $f\in \mcW^0_\lambda$, the leading term $f_0$ in the asymptotic expansion $f\sim \sum f_j$ defines a function on the sphere $SV^*$. It is the principal symbol in the Weyl calculus.
We denote this leading term $f_0$ by $\sigma^W(f)$. We have for all $\lambda\in \RR$,
\[ \sigma^W(f\#_\lambda g) = \sigma^W(f)\sigma^W(g)\qquad f,g\in \mcW^0_\lambda\]

\begin{example}\label{ex:harmonic osc}
Let $Q_\lambda\in \Wp_\lambda$
be the element determined by the polynomial $Q_\lambda(z)=\|z\|^2$ for $z\in V^*$.
With  notation as above
\[ Q_\lambda(x,\xi)=\sum_{j=1}^n (x_j^2+\xi_j^2)\qquad Q_\lambda =\sum_{j=1}^n (e_j^2+f_j^2)\]
If $\lambda\ne 0$,  Weyl quantization represents $Q_\lambda$ by
\[ \pi_\lambda(Q_\lambda) = \sum_{j=1}^n \left(x_j^2-\lambda^2\frac{\partial^2}{\partial x_j^2}\right) \]
With a change of coordinates $x_j=y_j\sqrt{|\lambda|}$ this is 
\[ |\lambda| \sum_{j=1}^n \left( y_j^2-\frac{\partial^2}{\partial y_j^2} \right)=|\lambda| Q\]
where $Q$ is the standard harmonic oscillator on $\RR^n$.

$Q$ is essentially self-adjoint, with spectrum $\{n, n+2, n+4,\dots\}$. The eigenspace for the smallest eigenvalue $n$ is $1$-dimensional, spanned by the Gaussian $e^{-\|x\|^2/2}$ (the ground state or vacuum vector).
If $s_\lambda\in \mcW_\lambda$ is the idempotent of Example \ref{composition example},
then $\pi_\lambda(s_\lambda)$ is the projection
onto the eigenspace  of  $\pi_\lambda(Q_\lambda)$ with eigenvalue $|\lambda| n$ (the vacuum projection).

As $\lambda\to 0$ the spectrum of $Q_\lambda=|\lambda| Q$ converges to the positive real line $[0,\infty)$,
which is the range of the polynomial $Q_0\in \Wp_0$.
\end{example}

\section{Equivariant Bott Periodicity as Deformation Quantization}\label{sec:Bott per and quantization}

Throughout this section,  if $X\cong \RR^m$ is a finite dimensional real vector space,
$BX$ is the radial compactification of $X$, obtained by adjoining the sphere $SX\cong S^{m-1}$ at infinity.
At times we shall think of $BX$ as a smooth manifold with boundary.
The smooth structure at the boundary is determined by taking $1/\|x\|$ (for $x\in X$) to be a smooth coordinate transversal to the boundary $SX$.

\subsection{Quantization}\label{sec:Cstar for Weyl algebra}

We denote by  $W^*_\lambda$ the $C^*$-algebra that is the norm completion of the order zero Weyl algebra $\mcW^0_\lambda$.
If $\lambda\ne 0$, the norm of $f\in W^*_\lambda$ is the operator norm $\|\pi_\lambda(f)\|$.
If $\lambda=0$, the norm of $f\in W^*_0$ is the supremum norm $\|f\|_\infty$ of functions on $BV^*$.
Then $W^*_0=C(BV^*)$.
The $C^*$-algebras $W^*_\lambda$ are the fibers of  a continuous field $\{W^*_\lambda\}_{\lambda\in [0,1]}$  parametrized by $\lambda\in [0,1]$. 
The continuous field  $\{W^*_\lambda\}_{\lambda\in[0,1]}$ is  trivial away from $0$.
We denote $W^*=W^*_1$. For $\lambda>0$ we have isomorphisms 
\[ \phi_\lambda: W^*\cong W^*_\lambda\quad \phi_\lambda(f)(x) = \frac{1}{\lambda^n}f\left(\frac{x}{\sqrt{\lambda}}\right)\]
(See \cite{Ri89} for details).

We denote by 
\[ C^*(V,\lambda\omega)\]
the norm completion of $\mcS_\lambda=\mcS(V^*,\#_\lambda)\cong\mcS(V,\ast_\lambda)$.
It is a two-sided ideal in $W^*_\lambda$.
If $\lambda=0$ then $C^*(V,\ast_0)$ is the convolution $C^*$ -algebra of the abelian group $V$,
and we have
\[C^*(V) = C_0(V^*)\]
If $\lambda\ne 0$ then $C^*(V,\lambda\omega)$ is isomorphic, via the representation $\pi_\lambda$,  to the $C^*$-algebra of compact operators on $L^2(\RR^n)$.

For each $\lambda\in\RR$ (including $\lambda=0$) there is a short exact sequence of $C^*$-algebras,
 \begin{equation}\label{eqn:weyl short exact sequence}
    0\to C^*(V,\lambda\omega)\to W^*_\lambda\stackrel{\sigma^W}{\lra} C(SV^*)\to 0
 \end{equation} 
The quotient $C(SV^*)$ is the commutative $C^*$-algebra of continuous functions on $SV^*$.

\subsection{The Bargmann-Fock representation}

The Bargmann-Fock space $H^{BF}$  is the Hilbert space  of  square-integrable holomorphic functions on the complex vector space $V^{0,1}$ (the dual space of $V^{1,0}$), with inner product
\[ \ang{f,g} = \frac{1}{\pi^n}\int f(z)\overline{g(z)} \,e^{-\|z\|^2} dz.\]
The representation $\pi_{BF}$ of the polynomial Weyl algebra $\Wp$ on $H^{BF}$ is given by
\[ \pi_{BF}(z_j) = z_j \qquad  \pi_{BF}(\bar{z}_j) = \frac{\partial}{\partial z_j} \qquad \pi_{BF}(1) = 1 .\]
The multiplication operator $z_j$ and complex derivative $\partial/\partial z_j$
are unbounded linear operators on $H^{BF}$.
The subspace of polynomials ${\rm Sym}\,\CC^n$ is  invariant  for both operators.
The Bargmann-Fock representation extends to the tensor algebra $\bigotimes V$, and then restricts to $\Wp$ (where we identify polynomials with symmetric tensors).
The operators $z_j$ and $\partial/\partial z_j$ are adjoints, and in general $\pi_{BF}(f)^*=\pi_{BF}(\bar{f})$ for all   $f\in \Wp$.

 The Bargmann-Fock representation extends to the full Weyl algebra $\mcW$, and is unitarily equivalent, as a representation of $\mcW$, to Weyl quantization by operators on $L^2(\RR^n)$.
 For details on the Bargmann integral transform, see \cite{Ba61}.
 
 The action of the unitary group  $U=U(V^{1,0})$  on $V^{1,0}$
 determines an action on $H^{BF}$.
 The Bargmann-Fock representation gives a $U$-equivariant $\ast$-isomorphism of $C^*$-algebras,
 \[ C^*(V,\omega) \cong \KK(H^{BF})\]
 In particular, we have a canonical isomorphism,
 \[ K_0^U(C^*(V,\omega))\cong R[U(n)]\]
 where $R[U(n)]$ is the representation ring of the unitary group $U(n)$.
 Note that we may identify $R[U(V^{1,0})]=R[U(n)]$,
 since any two isomorphisms $U(V^{1,0})\cong U(n)$ determined by a choice of basis for $V^{1,0}$
differ by an inner automorphism.

\subsection{Clifford multiplication}
Let $\Lambda$ be the $2^n$-dimensional complex vector space 
\[ \Lambda=\Lambda^\bullet V^{1,0} = \bigoplus_{q=0}^n \Lambda^q V^{1,0}\]
For $z\in V^{1,0}$ we denote by $\varepsilon_z, \iota_z\in \End(\Lambda)$ the linear maps
\[
\varepsilon_z \colon \Lambda^p V^{1,0} \to \Lambda^{p+1}V^{1,0}\qquad \varepsilon_z(\alpha):= z \wedge \alpha 
\]
\[\iota_z \colon \Lambda^{p+1} V^{1,0} \to \Lambda^{p}V^{1,0}\qquad \iota_z(\alpha)=z\lefthalfcup \alpha\]
If  $v_1,\dots,v_q\in V^{1,0}$ then
\[ z\lefthalfcup (v_1\wedge \cdots\wedge v_q):=\sum_{j=1}^q (-1)^{j+1} v_1\wedge\cdots\wedge \ang{v_j,z}\wedge\cdots\wedge v_q\]
The maps $\varepsilon_z, \iota_z$ are adjoints, i.e.
\[ \ang{\varepsilon_z\alpha, \beta} = \ang{\alpha, \iota_z\beta}\qquad \alpha,\beta\in\Lambda\]
$\Lambda=\Lambda^+\oplus \Lambda^-$ is $\ZZ_2$-graded by
\[\Lambda^+=\bigoplus_{q\text{ even}} \Lambda^q V^{1,0}\qquad  \Lambda^-=\bigoplus_{q\text{ odd}} \Lambda^q V^{1,0}\]
Clifford multiplication by $z$ is
\[    c(z)=\varepsilon_z+\iota_z \in \End(\Lambda)\qquad c(z)^2=\|z\|^2\cdot I\]
For $(z,t)\in V\times\RR$ in the $(2n+1)$-dimensional real vector space $V\times \RR$, Clifford multiplication is
\begin{equation}\label{eq:clifford}
    c(z,t) = \begin{bmatrix}
        t&c(z)\\c(z)&-t
    \end{bmatrix}\in \End(\Lambda)\qquad z\in V,\;t\in \RR
\end{equation} 
with
\[ c(z,t)^2=(\|z\|^2+t^2)\cdot I=\|(z,t)\|^2\cdot I\]
Here, as in the rest of this section, the $2\times 2$ matrix in (\ref{eq:clifford}) corresponds to the direct sum $\Lambda^+\oplus \Lambda^-$.


\subsection{Quantization of Clifford multiplication}\label{sec:Bott quantization}

With $\Lambda=\Lambda^\bullet V^{1,0}$, let 
\[A_\lambda\in \Wp_\lambda \otimes \End(\Lambda)\]
be the element in the tensor algebra  defined by the linear function
\[
A_\lambda:V^*\to \End(\Lambda) \qquad A_\lambda(z):= c(z)
\]
Here we identify $V=V^*$ via the Euclidean inner product $g(v,w)=\omega(v,Jw)$.
Then $z\in V^*$ gives an element in $V=V^{1,0}$, so that we can define
\[ c(z) = \varepsilon_z+\iota_z\in \End(\Lambda)\]
As above, let $e_1,\dots, e_n$ be an orthonormal basis for the complex vector space $V^{1,0}$,
and let $f_{j}=Je_j$.
Note that 
\[\varepsilon_{f_j} = \varepsilon_{ie_j} = i\varepsilon_{e_j}\qquad \iota_{f_j}=\iota_{ie_{j}}=-i\iota_{e_j}\]
We get
\[ A_\lambda = \sum_{j=1}^n e_j\otimes (\varepsilon_{e_j}+\iota_{e_j})+\sum_{j=1}^n if_{j}\otimes (\varepsilon_{e_j}-\iota_{e_j})\]
Weyl quantization represents $A_\lambda$ by a first order differential operator on $\RR^n$, which acts on sections in the trivial vector bundle $\RR^n\times \Lambda$.
With $\lambda=1$ we get
\[ \pi_1(A_1) = \sigma + D  = 
\sum_{j=1}^n x_j \otimes (\varepsilon_{e_j}+\iota_{e_j})+ \sum_{j=1}^n \frac{\partial}{\partial x_j}\otimes (\varepsilon_{e_j}-\iota_{e_j})\]
The first term $\sigma$ is an endomorphism of the bundle $\RR^n\times \Lambda$,
\[ \sigma(x) = \sum_{j=1}^n x_j\otimes (\varepsilon_{e_j}+\iota_{e_j})=\varepsilon_x+\iota_x\in \End(\Lambda)\]
This endomorphism squares to $\sigma(x)^2=\|x\|^2$. 

The second term $D$ is the differential operator 
\[ D = d+d^*\qquad  d=\sum_{j=1}^n \frac{\partial}{\partial x_j}\otimes \varepsilon_{e_j}\qquad d^*=\sum_{j=1}^n -\frac{\partial}{\partial x_j}\otimes \iota_{e_j} \]
If we identify 
\[\Lambda^\bullet V^{1,0}=(\Lambda^\bullet\RR^n)\otimes_\RR \CC\]
then $d$ is the complexified de Rham operator and $d^*$  its formal adjoint.

\begin{lemma}\label{lem:Asq}
The square of $A_\lambda$ in $\Wp_\lambda \otimes \End(\Lambda)$ is
 \[ 
 A_\lambda^2 = Q_\lambda\otimes 1+1\otimes \lambda N
 \] 
 where $Q_\lambda(x)=\|x\|^2$ is as in Example \ref{ex:harmonic osc},
 and $N\in \End(\Lambda)$ acts on $\Lambda^kV^{1,0}$ as multiplication by $2k-n$.
\end{lemma}
\begin{proof}
Let $e_1,\dots, e_n$ be an orthonormal basis for the complex hermitian vector space $V^{1,0}$
and let $f_{j}=Je_j$. For convenience write $e_{n+j}=f_j$.
Then
\[ A_\lambda=\sum_{j=1}^{2n} e_j\otimes c(e_j)\]
We have $e_j\#_\lambda e_j=e_j^2$ and $c(e_j)^2=1$, while $c(e_j), c(e_k)$ anticommute if $j\ne k$.
Thus
\[ A_\lambda^2 = \sum_{j=1}^{2n} e_j^2\otimes 1 + \sum_{j<k} [e_j,e_k]_{\#_\lambda} \otimes c(e_j)c(e_k)\]
If $j<k$ then
\[ [e_j,e_k]_{\#_\lambda}=
\begin{cases}
    [e_j,f_j] = i\lambda&\text{if } k=j+n\\
    0&\text{otherwise}    
\end{cases}\]
We obtain
\[ A_\lambda^2 = Q_\lambda\otimes 1 + 1 \otimes \lambda N\]
with
\[ N = \sum_{j=1}^n ic(e_j)c(f_{j})\]
Now $f_{j}=Je_j$ is identified with $ie_j$ in $V^{1,0}$.
Then,
\[c(e_j)c(f_{j})=
(\varepsilon_{e_j}+\iota_{e_j})(i\varepsilon_{e_{j}}-i\iota_{e_{j}})
=i(-\varepsilon_{e_j}\iota_{e_{j}}+\iota_{e_j}\varepsilon_{e_j})\]
and so
\[ N=\sum_{j=1}^n \varepsilon_{e_j}\iota_{e_j}-\iota_{e_j}\varepsilon_{e_j}\in \End(\Lambda)\]
To prove the lemma we need to analyze how $N$ acts on $\Lambda^kV^{1,0}$.
Consider the elementary tensor $\alpha=e_{j_1}\wedge \cdots \wedge e_{j_k}$ with $1\le j_1<\cdots<j_k\le n$. Denote $J=\{j_1,\dots, j_k\}$.
Then
\[ N\alpha = \sum_{j\in J}\varepsilon_{e_j}\iota_{e_j}\alpha -\sum_{j\notin J}\iota_{e_j}\varepsilon_{e_j}\alpha =k\alpha - (n-k)\alpha =(2k-n)\alpha\]
Since elementary tensors of the form $e_{j_1}\wedge \cdots \wedge e_{j_k}$ span $\Lambda^kV^{1,0}$, we see that $N\alpha=(2k-n)\alpha$  for all $\alpha\in \Lambda^kV^{1,0}$.

\end{proof}

\subsection{Quantization of the Bott generator}

In general, if  $f\in \mcW_\lambda$ has  asymptotic expansion $f \sim \sum_{j=-m}^\infty f_j$
and the leading term $f_{-m}$  is an invertible function on $SV^*$,
then $f$ has an inverse in  $\mcW_\lambda$ modulo $\mcS_\lambda$.

The linear function $A_\lambda:V^*\to \End(\Lambda)$
 is invertible on $V^*\setminus\{0\}$. 
Therefore  $A_\lambda$ is invertible in $\mcW_\lambda\otimes\End(\Lambda)$ modulo $\Sch_\lambda\otimes \End(\Lambda)$.
An explicit inverse modulo $\mcS_\lambda\otimes \End(\Lambda)$ for $A_\lambda$ is
\begin{equation}\label{informal Y}
    Y_\lambda=g(A_\lambda)\qquad g(s) = \frac{1-e^{-s^2}}{s}
\end{equation}
One way to make this  precise is to use  the Weyl pseudodifferential calculus.
In Appendix \ref{sec:explicit family of idempotents} we give a detailed construction in the Weyl algebra which does not rely on  functional calculus.

$A_\lambda$ is odd with respect to the grading on $\Lambda$,
\[
A_\lambda= \begin{bmatrix} 0& A^-_\lambda\\
A^+_\lambda &0
\end{bmatrix}
\qquad
A_\lambda^2= \begin{bmatrix} A^-_\lambda A^+_\lambda &0\\
0&A^+_\lambda A^-_\lambda
\end{bmatrix}
\]
with 
\[A^+_\lambda \in \Wp_\lambda \otimes \Hom (\Lambda^+, \Lambda^-) \qquad
A^-_\lambda =(A^+_\lambda)^* \in \Wp_\lambda \otimes \Hom (\Lambda^-, \Lambda^+) \]
For each $\lambda$, $A_\lambda^+$ determines an element in algebraic $K$-theory,
 \[[A_\lambda^+]\in K_0(\mcW_\lambda, \mcS_\lambda)=K_0(\mcS_\lambda)\]
The inclusion $\mcS_\lambda \to C^*(V,\lambda\omega)$ gives elements
\[ [A_\lambda^+] \in K_0(C^*(V,\lambda\omega))\]
An explicit formula for the idempotents in $K_0(C^*(V,\lambda\omega)$ is obtained by applying (\ref{eq:K boundary}), using the parametrix $Y_\lambda$.
The resulting idempotents are equivariant for the canonical action of $U=U(V^{1,0})$,
and we obtain elements in equivariant $K$-theory
\[ [A_\lambda^+] \in K_0^U(C^*(V,\lambda\omega))\]

\begin{lemma}
With $\lambda=0$, the $K$-theory class
\[[A_0^+] \in K_0^U(C_0(V^*))\]
is the Bott generator of $K^0_U(V^*)$, represented by the formal difference of idempotents,
\[ e_\beta-\begin{bmatrix}0&0\\0&1\end{bmatrix}\in K_0^U(C_0(V^*))\]
where $e_\beta$ is the projection valued function
\begin{equation}\label{eq:Bott3}
    e_\beta:V^*\to \End(\Lambda)\qquad e_\beta(z) := \frac{1}{1+\|z\|^2}
\begin{bmatrix}
    1&c(z)\\c(z)&\|z\|^2
\end{bmatrix}
\end{equation}.
\end{lemma}
\begin{proof}
If we think of $V\times \Lambda^+$ and $V\times \Lambda^-$ as trivial vector bundles on $V$ with the diagonal action of $U$,
then $A^+_0$ is the equivariant vector bundle homomorphism $V\times \Lambda^+\to V\times \Lambda^-$
defined by Clifford multiplication $A_0(v) = c(v)$.
This is the Bott generator in compactly supported $K$-theory $K^0_U(V^*)$.

Formula (\ref{eq:K boundary}) gives the corresponding element in $K_0^U(C_0(V^*))$ if we use
\[ x=\frac{1}{\sqrt{1+\|z\|^2}}c(z)\qquad 
y= \begin{cases}
    (\sqrt{1+\|z\|^2}-1)c(z)^{-1}&\text{if }z\ne 0\\
    0&\text{if }z=0
\end{cases}\]
\end{proof}

\begin{remark}
Equivariant Bott periodicity is the isomorphism 
\[K^0_U(V^*) \cong R[U(n)]\]
Under this isomorphism, the Bott generator in $K^0_U(V^*)$ corresponds to the trivial representation in $R[U(n)]$.  See \cite{At68}.
\end{remark}

\begin{lemma}
If $\lambda>0$, the $K$-theory class
\[[A_\lambda^+] \in K_0^U(C^*(V,\lambda\omega))\cong R[U(n)]\]
corresponds to the trivial one-dimensional representation of $U(n)$.
For $\lambda<0$ it is the determinant representation.
\end{lemma}

\begin{proof}
By  Lemma \ref{lem:Asq},
\[ \pi_\lambda(A_\lambda^2) = \pi_\lambda(Q_\lambda) \otimes 1 + 1\otimes \lambda N\]
If $\lambda\ne 0$ then $\pi_\lambda(Q_\lambda)$ is equivalant to $|\lambda|Q$, where $Q$ is the harmonic oscillator on $\RR^n$ (see Example \ref{ex:harmonic osc}).
$Q$ is a self-adjoint operator with spectrum $\{n, n+2, n+4, \cdots\}$.
The eigenspace of $Q$ with eigenvalue $n$ is one dimensional,
and is spanned by the Gaussian $e^{-\|x\|^2/2}$.

If $\lambda>0$, $\pi_\lambda(A_\lambda^+)$ is an (unbounded) Fredholm operator with zero cokernel and $1$-dimensional  kernel  spanned by the $0$-form $e^{-\|x\|^2/2}$.
The unitary group $U=U(V^{1,0})$ acts trivially on the kernel.
Thus, the equivariant index of  $\pi_\lambda(A^+_\lambda)$ is the trivial representation of $U$.

To apply (\ref{eq:K boundary}) we need a parametrix for $A^+_\lambda$.
The positive operator $\pi_\lambda(A^-_\lambda A^+_\lambda) \ge 0$ has discrete spectrum.
Let $f$ be a holomorphic  function defined in a neighborhood of the spectrum that is $0$ in a neighborhood of $0$, and $1/z$ in a neighborhood of the positive spectrum.
Applying holomorphic functional calculus,
we obtain a Weyl pseudodifferential operator $f(\pi_\lambda(A^-_\lambda A^+_\lambda))$ of order $-2$.
Then the order $-1$ operator 
\[f(\pi_\lambda(A^-_\lambda A^+_\lambda))\pi_\lambda(A^-_\lambda) \]
is a parametrix for $\pi_\lambda(A_\lambda^+)$ with corresponding Weyl symbol in $\mcW^{-1}_\lambda\otimes \Hom(\Lambda^-,\Lambda^+)$.
Then the idempotent $e$ in formula (\ref{eq:K boundary}) is 
\[ e=\begin{bmatrix} s&0\\0&1 \end{bmatrix}\]
where $s$ is projection onto the kernel of $\pi_\lambda(A^+)$.

If $\lambda<0$ the kernel of $\pi_\lambda(A_\lambda^+)$ is spanned by $e^{-\|x\|^2/2}$ times the $n$-form  $e_1\wedge \cdots\wedge e_n$.
An element $g\in U$ acts by the determinant $\det{g}$.

\end{proof}

\begin{proposition}\label{prop:bott to vacuum}
The  map in $K$-theory,
\[ K^0_U(V^*)=K_0^U(C^*(V))\to K_0^U(C^*(V,\omega))\cong R[U(n)]\]
determined by the deformation $\{C^*(V,\lambda\omega)\}_{\lambda\in [0,1]}$
maps the Bott generator  to the trivial representation of $U(n)$.
\end{proposition}
\begin{proof}
With the family of parametrices $Y_\lambda$ as in (\ref{informal Y}),
the family of idempotents corresponding to  $A_\lambda^+$ defined by (\ref{eq:K boundary}) is a continuous section in the field 
\[\{C^*(V,\lambda\omega)\otimes\End(\Lambda)\}_{\lambda\in [0,1]}\]
For a proof that does not rely on properties of the Weyl  pseudodifferential calculus, see Appendix \ref{sec:explicit family of idempotents}.

\end{proof}

\subsection{The equivariant Thom isomorphism}

Let $G$ be a compact Lie group, 
and let $E$ be a $G$-equivariant symplectic vector bundle on a compact smooth $G$-manifold $X$.
Assume that a compatible $G$-equivariant complex structure $J$ has been chosen in the fibers of $E$.

We denote by $C^*(E,\omega)$  the $C^*$-algebra of continuous sections in the field of $C^*$-algebras  $\{C^*(E_x,\omega_x)\}_{x\in X}$. 
Alternatively, $C^*(E,\omega)$ is the $C^*$-algebra for the groupoid $E$ (with abelian fibers $E_x$)   twisted by the groupoid $2$-cocycle 
\[E^{(2)}=E\times_X E\to U(1)\qquad (v,w)\mapsto \exp(i\omega_x(v,w))\qquad v,w\in E_x\]
The family of Bargmann-Fock spaces $\{H_x^{BF}\}_{x\in X}$ associated to the fibers $E_x\cong \CC^n$
form a continuous field of Hilbert spaces.
Let $\mcE^{BF}$ be the corresponding  Hilbert $C(X)$-module.
The Bargmann-Fock representations of the $C^*$-algebras $C^*(E_x,\omega_x)$ on the Hilbert spaces $H_x^{BF}$ determine an isomorphism,
\[ C^*(E,\omega)\cong \KK(\mcE^{BF})\]
Thus, we have Morita equivalence $C^*(E,\omega)\sim C(X)$.
The Morita equivalence bimodule $\mcE^{BF}$ is $G$-equivariant.

Deforming the symplectic form $\omega$,
the family $\{C^*(E,\lambda \omega)\}_{\lambda\in [0,1]}$
is  a continuous field of $C^*$-algebras that is trivial away from $0$.
At $\lambda=0$ we have $C_0(E^*)$.

\begin{proposition}\label{prop:Thom}
The map in equivariant $K$-theory
\[ K^0_G(E^*)\to K_0^G(C^*(E,\omega))\cong K^0_G(X)\]
determined by the deformation $\{C^*(E,\lambda \omega)\}_{\lambda\in [0,1]}$
maps the Thom class of the complex vector bundle $E$ to the trivial line bundle on $X$.
\end{proposition}

\begin{proof}
This follows immediately from Proposition \ref{prop:bott to vacuum}.
\end{proof}

\section{Deformation From the Dirac Symbol to the Szeg\"o Symbol}\label{sec:symbols in K theory}

\subsection{Contact manifolds}

Throughout this paper, $M$ is a compact $(2n+1)$-dimensional smooth manifold equipped with a contact form $\theta$.
A contact form $\theta$  is a differential $1$-form with the property that $\theta(d\theta)^n$ is a volume form, i.e. a nowhere vanishing $2n+1$-form.
We choose the orientation of $M$ such that 
\[\int_M \theta(d\theta)^n>0\]
Let $H\subset TM$ be the vector bundle of tangent vectors in the kernel of $\theta$.
$\theta$ is a contact form if and only if $d\theta$ restricts to a symplectic form in  the fibers of $H$,
\[ \omega = d\theta|_{H}\]
Throughout we assume that a complex structure $J$ has been chosen for $H$ that is compatible with   $\omega$, i.e.
\[\omega(Jv,Jw)=\omega(v,w)\qquad \omega(v,Jv)>0\;\text{ for } v\ne 0\]
Then the fibers of $H$ are complex  vector spaces, where $i$ acts as $J$.
We denote this complex vector bundle by $H^{1,0}$.

The Reeb  field $R$ is the unique vector field on $M$ with $d\theta(R,\,.\,)=0$ and $\theta(R)=1$.
The Reeb field trivializes the normal bundle $TM/H\cong \underline{\RR}$.
We get an isomorphism
\[ TM\cong H^{1,0}\oplus \underline{\RR}\]
Thus, $M$ is a stably almost complex manifold, and in particular $M$ is a Spin$^c$ manifold.
The Spin$^c$ structure for $TM$ is the direct sum of the canonical Spin$^c$ structure of the complex vector bundle $H^{1,0}$ and the trivial line bundle $\underline{\RR}=M\times\RR$.
The spinor bundle of $M$ is $\Lambda^\bullet H^{1,0}$.

The Todd class $\Td(M)$ of a contact manifold $M$ is the Todd class of the complex vector bundle $H^{1,0}$,
\[ \Td(M)=\Td(H^{1,0})\]
We assume that a compact Lie group $G$ acts smoothly on $M$, preserving the contact form $\theta$.
Then the symplectic form $\omega=d\theta|_{H}$ is $G$-invariant.
We choose the compatible complex structure $J$ to be also $G$-invariant, making $H^{1,0}$ a $G$-equivariant complex vector bundle.

\begin{lemma}
There exists a $G$-invariant complex structure $J$ on $H$ that is compatible with $\omega=d\theta|_H$.
\end{lemma}
\begin{proof}
Choose a $G$-invariant positive definite inner product $q$ in the fibers of $H$.
Then with respect to $q$ we have 
\[\omega(v,w)=q(v,A_pw)\qquad v,w\in H_p,\,p\in M\] 
where $A_p\in \End(H_p)$ depends smoothly on $p$.
Since $\omega$ is nondegenerate, $A_p$ is invertible.
From $\omega(v,w)=-\omega(w,v)$ we get  $A^T=-A$, where the transpose is with respect to $q$.
Now let $J=-A|A|^{-1}$ with $|A|=\sqrt{A^TA}=\sqrt{-A^2}$.
Then $J^2=A^2|A|^{-2}=-I$.
Since $\omega, q$ are $G$-invariant, so is $A$ and therefore also $J$.
We have $J^T=-|A|^{-1}A^T=-J=J^{-1}$, i.e. $J$ is orthogonal for $q$. 
Then
\[\omega(Jv,Jw)=q(Jv,AJw)=q(Jv,JAw)=q(v,Aw)=d\theta(v,w)\]
Finally, $\omega(v,Jv)=q(v,-A^2|A|^{-1}v)=q(v,|A|v)>0$ if $v\ne 0$.
Thus, $J$ is compatible with $\omega$. 

\end{proof}

\subsection{Heisenberg calculus}

This section serves to fix notation and conventions.
For a quick introduction to the Heisenberg calculus, see section 2 in \cite{GvE22}.
For full details see \cites{Ta84, BG88, CGGP92, EMxx, EM98, Po08}.
\vskip 6pt
For a Heisenberg pseudodifferential operator $\mcL$ of order $d$ on a contact manifold $M$,
the principal symbol is a pair
\[\sigma^d_H(\mcL)=(\sigma_+(\mcL),\sigma_-(\mcL))\]
of smooth complex-valued functions $\sigma_\pm(\mcL)$ on the total space of $H^*$.
For each $x\in M$, the restriction of  $\sigma^d_H(\mcL)$ to the fiber $H^*_x$ is a pair of elements of order $d$ in the Weyl algebras associated to the symplectic vector spaces $(H^*_x,\omega)$ and $(H^*_x,-\omega)$ with $\omega=d\theta$,
\[ \sigma^d_+(\mcL, x)\in \mcW^d(H^*_x, \omega)\qquad  \sigma^d_-(\mcL,x)\in \mcW^d(H^*_x,-\omega)\]
The product in $\mcW(H^*_x,\omega)$ is the Moyal product $\#=\#_1$ with $\lambda=1$, as in Definition \ref{def:moyal product} . The product in $\mcW(H^*_x,-\omega)$ is the Moyal product $\#_{-1}$ with $\lambda=-1$.

If $X\in \Gamma(H)$ is a vector field on $M$ tangent to $H$,
then $X$ is a differential operator of order $1$, with principal Heisenberg symbol
\[ \sigma^1_H(X) = (iX,iX)\]
where $iX$ is viewed as a linear function in the fibers of $H^*$.
The Reeb field $R$ on $M$ is a differential operator of order two  in the Heisenberg calculus, with principal symbol
\[ \sigma_H^2(R) = (i,-i)\]
From here on we restrict ourselves to the algebra of order zero Heisenberg pseudodifferential operators.
The order zero Weyl algebra elements $\sigma_+(\mcL), \sigma_-(\mcL)$ have  asymptotic expansions which contain only {\em even} terms,
\begin{equation*} \sigma_\pm\sim \sum_{j=0}^\infty (\sigma_\pm)_{2j}\end{equation*}
The asymptotic  expansions of $\sigma_+(\mcL,x)$ and $\sigma_-(\mcL,x)$ are related by 
\[
(\sigma_-)_{2j} = (-1)^j(\sigma_+)_{2j}
\]
In particular, the leading terms are equal, $(\sigma_+)_0=(\sigma_-)_0$.

\vskip 6pt
We denote the algebra of Heisenberg principal symbols of order zero by $\SymH^0$.
Elements in $\mcW^0(H^*_x,\pm\omega)$ are represented, by the Weyl calculus, as bounded operators on Hilbert space.
We define the norm of an element $\sigma=(\sigma_+,\sigma_-)\in \SymH^0$
to be the supremum  of the operator norms $\|\sigma_+(x)\|$, $\|\sigma_-(x)\|$
over all $x\in M$.
Then $\SymH^*$ is the $C^*$-algebra obtained as the norm completion of $\SymH^0$.

\subsection{Szeg\"o projections}

Boutet de Monvel's theorem, as well as the Baum-Douglas-Taylor theorem, assume that $M=\partial \Omega$ is the strictly pseudoconvex boundary of a complex domain $\Omega$.
The Szeg\"o projection $S$ is 
the projection of $L^2(M)$ onto the Hardy space $H^2(M)$,
which is the closed subspace of boundary values (in the non-tangential sense) of holomorphic functions on $\Omega$ with finite Hardy norm.

A strictly pseudoconvex boundary is a contact manifold.
Choose a defining function
\[\rho:\text{neighborhood of } \bar\Omega\to \RR\quad \Omega=\{\rho< 0\}\quad M=\{\rho=0\}\quad d\rho\ne 0 \text{ on }M\]
Then $\theta = -d\rho\circ J$ is a contact form on $M$.
Strict pseudoconvexity is equivalent to the condition that the Levi form $d\theta(\cdot,J\cdot)$ is positive definite on $H=\ker \theta$. In particular, $d\theta$ restricts to a symplectic form on $H$ and
   $J$  is compatible with $d\theta$. 

The Spin$^c$ structure determined by  the contact (CR) structure on $M$ agrees with the  Spin$^c$ structure induced on $M$  as the boundary of the complex domain $\Omega$.

The Szeg\"o projection $S$ is a pseudodifferential operator of order zero in the Heisenberg calculus on $M$.
The principal Heisenberg symbol of $S$ is $\sigma_H(S) = (\sigma_+(S), \sigma_-(S))$ with $\sigma_-(S)=0$ and
\begin{equation}\label{eqn:szego symbol1}
    \sigma_+(S)=s\qquad s(x,v)= 2^ne^{-\|v\|^2}\qquad x\in M, v\in H_x^*
\end{equation}
See Example \ref{composition example}.

Not every contact manifold has a complex filling.
Following Epstein and Melrose  \cite{EMxx}, a generalized Szeg\"o projection on a contact manifold is any order zero pseudodifferential operator $S$ in the Heisenberg calculus
with principal symbol $(s,0)$ as in (\ref{eqn:szego symbol1}), and such that $S=S^*=S^2$.
To obtain a projection $S$, choose an arbitrary operator with symbol $(s,0)$ and apply holomorphic functional calculus. 

If a compact Lie group $G$ acts on $M$, the principal symbol $\sigma_H(S)$ is $G$-invariant.
We can choose $S$ to be $G$ invariant.

The principal Heisenberg symbol of the Szeg\"o projection $S$ determines an element in $K$-theory,
\[ [\sigma_H(S)]\in K_0^G(\SymH^*)\]



\subsection{The symbol of the Dirac operator}\label{sec:symbol of Dirac}

A contact manifold is stably almost complex, and hence Spin$^c$.
The Spin$^c$ structure of $TM$ is the direct sum of the Spin$^c$ structure of the complex vector bundle $H^{1,0}$ and the trivial Spin$^c$ structure of the trivial line bundle $\underline{\RR}$.
As an odd dimensional Spin$^c$ manifold, $M$ has a self-adjoint Dirac operator $D$, acting on sections of the (ungraded) spinor bundle $\slashed{S}=\Lambda^\bullet H^{1,0}$.
The Spin$^c$ structure of $M$ is $G$-invariant, and we can choose $D$ to be $G$-equivariant.

The principal symbol of the Dirac operator $D$ is 
\[
    \sigma(D)(x,\xi)= c(\xi)\qquad (x,\xi)\in S^*M
\]
Clifford multiplication $c(\xi)$ is as in (\ref{eq:clifford}), with $V=H^{1,0}_x$.
If $\|\xi\|=1$ then $c(\xi)$ is self-adjoint and unitary.
In particular, $c(\xi)$ has eigenvalues $\pm 1$.

$D$ is  self-adjoint with discrete spectrum.
Let $P_+$ be orthogonal projection of $L^2(M,\Lambda^\bullet H^{1,0})$
onto the direct sum of the eigenspaces of $D$ for  positive eigenvalues.
If $\|\xi\|=1$ then $\sigma(P_+)(x,\xi)$ is projection onto the $+1$ eigenspaces of $c(\xi)$,
\[ \sigma(P_+)(x,\xi)=\frac{1}{2}(1+c(\xi))\qquad (x,\xi)\in S^*M\]
The projection valued function $\sigma(P_+)$ determines a class in $K$-theory,
\[ [\sigma(P_+)]\in K^0_G(S^*M)\]
Let $F$ be the bounded operator
\[F=D(1+D^2)^{-1/2}\]
$F$ is an  elliptic pseudodifferential operator of order zero.
The full symbol of  $F$ determines an element
\[ [\sigma(F)]\in K^1_G(T^*M)\]
represented by the self-adjoint function
\[ f:T^*M\to \End(\Lambda^\bullet H^{1,0})\qquad f(\xi)=\chi(c(\xi))\]
with $\chi(s)=s(1+s^2)^{-1/2}$.

The two $K$-theory classes are related via the boundary map in the six-term exact sequence for the pair of compact spaces $(B^*M, S^*M)$,
\[ \partial:K^0_G(S^*M)\to K^1_G(T^*M)\qquad \partial([\sigma(P_+)])=[\sigma(F)]\]
See  (\ref{eq:K boundary 3}) in the appendix.

\subsection{Stereographic projection}\label{sec:Bott2}

Stereographic projection of the sphere $S(V\times\RR)$ onto $V$ through the South pole $p_s=(0,-1)\in S(V\times\RR)$ is
\[ \rho:S(V\times\RR)\setminus \{p_s\}\to V\qquad \rho(z,t)= \frac{z}{1+t} \qquad z\in V,\;t\in \RR\]
The map $\rho$ extends to the one-point compactification of $V$,
\[ \rho:S(V\times\RR)\to V^+=V\sqcup\{p_\infty\}\qquad \rho(p_s)=p_\infty\]
Note that the function $e_\beta$ of (\ref{eq:Bott3}) extends to $V^+$, since 
\[ \lim_{\|z\|\to \infty} e_\beta(z) = \begin{bmatrix}0&0\\0&1\end{bmatrix}\]

\begin{lemma}\label{lem:stereo}
Pullback under  stereographic projection,
\[\rho^*:K^0_U(V^+)\to K^0_U(S(V\times\RR))\]
gives $\rho^*(e_\beta)= e$,
where $e$ is the projection valued function
\begin{equation}\label{eq:Bott1}
e:S(V\times\RR)\to \End(\Lambda^\bullet V^{1,0})\qquad 
 e(z,t) = \frac{1}{2}(1+c(z,t))
 \end{equation}
\end{lemma}
\begin{proof}
A straightforward calculation of the composition $e_\beta\circ \rho$ gives    
\[
 e(z,t) = \frac{1}{2}\begin{bmatrix}
    1+t&c(z)\\c(z)&1-t    
    \end{bmatrix}\qquad (z,t)\in S(V\times \RR)
\]
Using (\ref{eq:clifford}) we obtain (\ref{eq:Bott1}).

\end{proof}

\begin{lemma}\label{lem:stereo in K}
Under the  isomorphism
\[ \rho^*:K^0_G(H^*)\to K^0_G(S^*M)\]
the symbol $[\sigma(P_+)]$ of the spectral projection $P_+$ of the Dirac operator is the image of the Thom class of the symplectic  vector bundle $H^*$. This Thom class is
\[ e -
\begin{bmatrix}0&0\\0&1\end{bmatrix}\in K^0_G(C_0(H^*))\]
where $e$ is the projection valued function
\[ e(x,z) = \frac{1}{1+\|z\|^2}
\begin{bmatrix}1&c(z)\\c(z)&\|z\|^2\end{bmatrix}\qquad x\in M,\;z\in H^*_x\]
\end{lemma}
\begin{proof}
Lemma  \ref{lem:stereo} applies fiber-by-fiber, with $V=H^{1,0}_x$.
\end{proof}

\begin{proposition}\label{prop:Dirac to Szego in even K theory}
The map in $K$-theory
determined by the deformation $C(S^*M)\rightsquigarrow \SymH^*$
maps the symbol of the spectral projection $P_+$ of the Dirac operator $D$  to the Heisenberg symbol of the Szeg\"o projection $S$,
\[  K^0_G(S^*M)\to K_0^G(\SymH^*)\qquad [\sigma(P_+)]\mapsto [\sigma_H(S)]\]
\end{proposition}
\begin{proof}
We have the   commutative diagram,
\[
\xymatrix{
K^0_G(H^*)\ar[r]^-{\rho^*}\ar[d]^{\ast_\lambda} 
&K^0_G(S^*M)\ar[d]\\
K_0^G(C^*(H,\omega))\ar[r]
&K_0^G(\SymH^0)
}
\]
where the horizontal map at the bottom maps a smooth family  $e\in \mcS(H,\omega)$ of idempotents $e(x)\in \mcS(H_x,\omega)$ to the Heisenberg symbol $(e(x),0)$.
By Lemma \ref{lem:stereo in K}, the top row  maps the Thom class of $H^{1,0}$ to $\sigma(P_+)$.
By Proposition \ref{prop:Thom}, the vertical arrow on the left maps the Thom class to the family of vacuum projections $s(x)\in \mcS(H_x^*,\#)\cong \mcS(H_x,\omega)$.
The bottom row maps this family $s$ to the principal Heisenberg symbol $\sigma_H(S)=(s,0)$ of the Szeg\"o projection.

\end{proof}

\section{From Symbols to Operators}\label{sec:symbols to operators}

\subsection{Full symbols on the Heisenberg group}\label{sec:Heisenberg group quantization}

We construct a unitalization of the group $C^*$-algebra $C^*(\Heis)$ of the Heisenberg group $\Heis=V\times \RR$, with notations as in section \ref{sec:Heisenberg and Weyl}.

For $\lambda>0$, denote by $\delta_\lambda$ the automorphism of $\Heis$ defined by
\[ \delta_\lambda(v,t)=(\lambda v,\lambda^2t)\qquad v\in V, t\in \RR\]
For a smooth function $f\in C^\infty(\Heis)$ 
define $(\delta^*_\lambda f)=f\circ \delta_\lambda$.
Extend $\delta_\lambda$ to distributions.
A distribution is homogeneous of degree $m>0$ if $\delta_\lambda^*u=\lambda^mu$ for all $\lambda>0$.

Let $\Psi^0(\Heis)$ be the convolution algebra
whose elements are compactly supported distributions $u\in \mcE'(\Heis)$ with the properties:
\begin{itemize}
    \item The restriction of $u$ to $\Heis\setminus \{0\}$ is a smooth function
    \item For any $\lambda>0$ the difference $\delta^*_\lambda u-u$ is smooth at $0\in G$
\end{itemize}
Convolution of distributions on $\Heis$  gives a representation
\[ \Psi^0(\Heis)\to \mcB(L^2(\Heis))\qquad u\mapsto [\psi\mapsto u\ast \psi]\]
We denote by $\Psi^*(\Heis)$ the $C^*$-algebra that is the completion of $\Psi^0(\Heis)$ in the operator norm.

There is a short exact sequence
\begin{equation}\label{eq:ses Heisenberg group}
    0\to C^*(\Heis)\to \Psi^*(\Heis)\to \SymH(\Heis)\to 0
\end{equation} 
The quotient  $\SymH(\Heis)$ is the $C^*$-algebra
\[ \SymH(\Heis) = \{(a,b) \in W^*_{+1}\oplus W^*_{-1}\mid \pi^W(a)=\pi^W(b)\}\]

\subsection{The Heisenberg tangent space}

With the identification $TM=H\oplus \underline{\RR}$
and symplectic form $\omega=d\theta|H$ in the fibers of $H$,
the fibers of the tangent space $TM$ have the structure of a Heisenberg group $\Heis_x=H_x\times \RR$.

The group structures in the fibers of $TM$ make the total space into a Lie groupoid, which we denote by $T_HM$.
As a smooth manifold $T_HM=TM$. Algebraically, the groupoid $T_HM$ is a disjoint union of a collection of  Heisenberg groups, 
\[ T_HM = \bigsqcup_{x\in M} \Heis_x\]
$C^*(T_HM)$ is the groupoid convolution $C^*$-algebra of $T_HM$.

$C^*(T_HM)$ is the $C^*$-algebra of continuous sections in a continuous field $\{C^*(\Heis_x)\}_{x\in M}$.
The fiber  $C^*(\Heis_x)$ is the group $C^*$-algebra of the Heisenberg group $\Heis_x$.
The continuous field structure is defined by declaring elements in $C_c^\infty(T_HM)$ to be continuous sections.

Let $\Psi^*(T_HM)$ be the $C^*$-algebra of continuous sections in the field $\{\Psi^*(\Heis_x)\}_{x\in M}$.
From (\ref{eq:ses Heisenberg group}) we get the short exact sequence
\begin{equation}\label{eq:seq for H}
    0\to C^*(T_HM)\to \Psi^*(T_HM)\to \SymH^*\to 0
\end{equation}

\subsection{The symbol of the Szeg\"o projection}

Let $\mcL$ be an order zero Heisenberg pseudodifferential operator on the compact contact manifold $M$.
If $\mcL$ has invertible principal Heisenberg symbol $\sigma_H(\mcL)$ then $\mcL$ is a hypoelliptic Fredholm operator.
The Heisenberg principal symbol of $\mcL$ is an invertible element in the algebra $\SymH^0$, and therefore determines a $K$-theory element,
\[[\sigma_H(\mcL)]\in K_1(\SymH^*)\]
Via the boundary map for the short exact sequence (\ref{eq:seq for H}),
we obtain an element 
\[\partial([\sigma_H(\mcL)])\in K_0(C^*(T_HM))\]
A slightly different construction of this element in $K_0(C^*(T_HM))$ was first proposed in \cite{vE10a}.

Now let $S$ be a $G$-equivariant Szeg\"o projection on $M$, and  $F=2S-1$ the self-adjoint operator
that is $+1$ on the range of $S$ and $-1$ on the complement of the range.
Using (\ref{eq:K boundary 3}) we have
\[ [\sigma_H(S)]\in K_0^G(\SymH^*)\qquad \partial([S])=[\sigma_H(F)]\in K_1^G(C^*(T_HM))\]

\subsection{Quantization on the cotangent space}

\begin{proposition}\label{prop:Dirac to Szego in odd K theory}
The quantization map
\[ K^1_G(T^*M)\to K_1^G(C^*(T_HM))\]
takes the class in $K^1_G(T^*M)$ represented by the symbol of the Dirac operator $D$ of $M$ 
to the class in $K_1^G(C^*(T_HM))$ represented by the Heisenberg symbol of a Szeg\"o projection $S$ on $M$.
\end{proposition}

\begin{proof}
Deformation of the group law for $\Heis=V\times\RR$ as
\[ (v,t)\cdot (w,s) = (v+w, t+s-\tfrac{1}{2}\lambda\omega(v,w))\]
(replacing $\omega$ by $\lambda\omega$) with $\lambda\in [0,1]$ gives a deformation of short exact sequences,
 \begin{equation*}
\xymatrix{
0\ar[r]
&C_0(\Lh^*)\ar[r]\ar@{.>}[d]^{\delta_\lambda} 
&C(B(\Lh^*))\ar[r]\ar@{.>}[d]^{\delta_\lambda}
&C(S(\Lh^*))\ar@{.>}[d]^{\delta_\lambda}\ar[r]
&0\\
0\ar[r]
&C^*(\Heis)\ar[r]
&\Psi^*(\Heis)\ar[r]
&\SymH(\Heis)\ar[r]
&0
}
\end{equation*}
Here $\Lh=V\times\RR$ is the Lie algebra of $\Heis$, $S(\Lh^*)$ is the sphere of parabolic rays in $\Lh^*$,
and $B(\Lh^*)=\Lh^*\sqcup S(\Lh^*)$ is the corresponding compactification of $\Lh^*$.

On a contact manifold $M$ we obtain
 \begin{equation*}
\xymatrix{
0\ar[r]
&C_0(T^*M)\ar[r]\ar@{.>}[d]^{\delta_\lambda} 
&C(B^*M)\ar[r]\ar@{.>}[d]^{\delta_\lambda}
&C(S^*M)\ar@{.>}[d]^{\delta_\lambda}\ar[r]
&0\\
0\ar[r]
&C^*(T_HM)\ar[r]
&\Psi^*(T_HM)\ar[r]
&\SymH^*\ar[r]
&0
}
\end{equation*}
from which we get the commutative diagram in $K$-theory,
\begin{equation*}
\xymatrix{
K^0_G(S^*M)\ar[r]^-{\partial}\ar[d]^{\delta_\lambda}
&K^1_G(T^*M)\ar[d]^{\delta_\lambda}\\
K_0^G(\SymH^*)\ar[r]^-{\partial}
&K_1^G(C^*(T_HM))
}
\end{equation*}
Proposition \ref{prop:Dirac to Szego in odd K theory} now follows from Proposition \ref{prop:Dirac to Szego in even K theory}.

\end{proof}

\subsection{$K$-homology}
The (odd) analytic $K$-homology of $M$,
\[ K_1(M)=KK^1(C(M),\CC)\]
is the abelian group of homotopy equivalence classes of triples $(H,\psi,F)$,
where
\begin{itemize}
\item $H$ is a Hilbert space
\item $\psi:C(M)\to \mcB(H)$ is a $\ast$-representation of $C(M)$ on $H$
\item $F\in \mcB(H)$ is a  self-adjoint bounded linear operator on $H$
\end{itemize}
satisfying the compactness relations
\begin{itemize}
    \item $F^2-I\in \KK(H)$
    \item $[F,\psi(f)]\in \KK(H)$ for all $f\in C(M)$
\end{itemize}
The even $K$-homology group $K_0(M)=KK_0(C(M),\CC)$ is defined similarly,
but now  $H=H^0\oplus H^1$ is $\ZZ_2$ graded,
$\psi(f)$ is even, and $F$ is odd. 
\begin{example}
An order zero pseudodifferential  operator $F$ on a compact manifold $M$
is bounded  on $H=L^2(M)$.
If $F$ is self-adjoint and elliptic, then $F^2-I$ is compact. 
Let $\psi:C(M)\to \mcB(L^2(M))$ be the representation of functions $f\in C(M)$ by multiplication operators $\psi(f)=M_f$.
For any $f\in C^\infty(M)$, the commutator $[F,\psi(f)]$ is a pseudodifferential operator of order $-1$, and hence compact as an operator on $L^2(M)$.
By density of $C^\infty(M)$ in $C(M)$, the same holds for $f\in C(M)$.
Thus, $(H,\psi,F)$ defines an element in $K_1(M)$.
\end{example}

\begin{example}\label{ex:dirac}
Let $D$ be the Dirac operator of an odd dimensional compact Spin$^c$ manifold $M$.
$D$ is self-adjoint.
Then $F=D(1+D^2)^{-1/2}$
is a self-adjoint elliptic order zero pseudodifferential  operator on  $M$.
It is a bounded operator on $H=L^2(M,\slashed{S})$,
where $\slashed{S}$ is the spinor bundle of $M$.
With $\psi(f)=M_f$ as above, $D$ determines an element in $K_1(M)$,
\[ [D] = [(L^2(M,\slashed{S}), \psi, D(1+D^2)^{-1/2})]\in K_1(M)\]
Let $P_+$ be the projection onto the direct sum of the eigenspaces of $D$ for positive eigenvalues,
and let $F'=2P_+-I$.
The difference $F-F'$ is a compact operator.
Therefore we may replace $F$ by $F'$.

\end{example}
\begin{example}\label{ex:szego}
The Szeg\"o projection $S$ is an order zero pseudodifferential operator in the Heisenberg calculus.
It is a bounded linear operator on $L^2(M)$.
Let $F=2S-I$. Since $S$ is a projection, $F$ is self-adjoint and $F^2=I$.
Commutators $[F,\psi(f)]$ are of order $-1$ in the Heisenberg calculus, and hence compact as operators on $L^2(M)$.
Thus, $S$ determines an element in $K_1(M)$,
\[ [S]=[(L^2(M), \psi, 2S-I)]\in K_1(M) \]
\end{example}

\subsection{Proof of the main theorem}

Let $\mathrm{Op}_H$ be the ``choose an operator'' map,
\[\mathrm{Op}_H:K_0(C^*(T_HM))\to K_0(M)\]
with the property that if $T$ is a Heisenberg pseudodifferential operator of order zero with invertible principal  symbol, then
\[ \mathrm{Op}_H(\partial([\sigma_H(T)]) = [T]\]
This is analogous to the ``choose an operator map'' for elliptic operators
\[ \mathrm{Op}_e:K^0(T^*M)\to K_0(M) \]
The existence of the map $\mathrm{Op}_H$ is shown in \cite{BvE14}.
Moreover, by Theorem 5.4.1 in \cite{BvE14}, the following diagram commutes,
\[
\xymatrix{
K^0(T^*M)\ar[r]^-{\delta_\lambda} \ar[dr]_{\mathrm{Op}_e}
&K_0(C^*(T_HM))\ar[d]^{\mathrm{Op}_H}\\
&K_0(M)
}
\]
The proof of Theorem 5.4.1 in \cite{BvE14}
applies, mutatis mutandis, to show the existence of a ``choose a self-adjoint operator'' map
\[ \mathrm{Op}_H:K_1(C^*(T_HM))\to K_1(M)\qquad \mathrm{Op}_H([\sigma_H(F)])=F\]
and commutativity in the diagram
\begin{equation}\label{eq:choose an operator}
    \xymatrix{
K^1(T^*M)\ar[r]^-{\delta_\lambda} \ar[dr]_{\mathrm{Op}_e}
&K_1(C^*(T_HM))\ar[d]^{\mathrm{Op}_H}\\
&K_1(M)
}
\end{equation}
If a compact Lie group $G$ acts on $M$
we obtain the analogous diagram in equivariant $K$-theory.


Theorem \ref{thm:main} follows  from Proposition \ref{prop:Dirac to Szego in odd K theory} and the commutative diagram (\ref{eq:choose an operator}).

\subsection{Index formula for Toeplitz operators}\label{sec:equivariant BdM formula}

With $M$ and $G$ as above, let  $\pi:G\to U(r)$ be a unitary representation of $G$,
and let $f:M\to \mathrm{GL}(r,\CC)$ be a $G$-equivariant smooth function, which means that \[f(g.p)=\pi(g)f(p)\pi(g)^{-1}\qquad g\in G,\;p\in M\]
Let $M_f$ be the multiplication operator
\[ (M_fu)(x) = f(x)u(x)\qquad u\in C^\infty(M,\CC^r)\]
We can form the Toeplitz operator,
\[T_f=(S\otimes I)\circ M_f:H^2(M)\otimes \CC^r \to H^2(M)\otimes \CC^r\]
where $S$ is a Szeg\"o projection. $T_f$ is a $G$-equivariant  Fredholm operator.

Before we can state the equivariant index formula for $T_f$, we need a  standard lemma. 

\begin{lemma}\label{lem:compact group acts on symplectic space}
Let $K$ be a compact group acting linearly on a symplectic vector space $V$, preserving the symplectic form $\omega$.
Then $\omega$ is nondegenerate when restricted to  the fixed set $V^K=\{v\in V\mid g.v=v\text{ for all }g\in K\}$.
\end{lemma}
\begin{proof}
Assume $u\in V^K$ is such that $\omega(u,w)=0$ for all $w\in V^K$.
Then for any $v\in V$ we have $\omega(u,v)=\omega(g.u, g.v)=\omega(u,g.v)$ for all $g\in K$, and therefore,
$\omega(u,v) = \omega(u,\bar{v})$ where $\bar{v}$ is the average
\[ \bar{v}=\int_K g.v\,dg\]
But $\bar{v}\in V^K$ and so by assumption $\omega(u,\bar{v})=0$.
It follows that $\omega(u,v)=0$ for all $v\in V$,
and therefore $u=0$ by nondegeneracy of $\omega$. 

\end{proof}

\begin{lemma}\label{lem:restriction of theta to M^g}
Let $G$ be a compact Lie group  acting smoothly on a  contact manifold $M$, preserving the contact form $\theta$.
Then $\theta$ restricts to a contact form on the fixed point set $M^g=\{p\in M\mid g.p=p\}$, for every $g\in G$.
\end{lemma}
\begin{proof}
Let $p\in M^g$, i.e. $g.p=p$.
The action of $g$ on $T_pM$ preserves  $\theta_p$ and  $d\theta_p$.
Then $g$ also preserves the Reeb vector $R_p\in T_pM$.
Moreover, $g$ preserves the subspace $H_p=\ker \theta_p$, and the symplectic form $d\theta_p$ on $H_p$.
Let $H_p^g$ be the fixed set for the action of $g$ on $H_p$.
The tangent space $T_pM^g$ is the direct sum of $H_p^g$ and the span of $R_p$.
To prove that $\theta$ restricts to a contact form on $M^g$ we need to prove that $d\theta_p$ is nondegenerate on $H_p^g$.
This follows from Lemma \ref{lem:compact group acts on symplectic space}.

\end{proof}

Note that the normal bundle $N^g$ of $M^g$ in $M$ is the quotient of the complex vector bundle $H^{1,0}$ (restricted to $M^g$) by the corresponding complex vector bundle for the contact manifold $M^g$. 
This is so because the Reeb field on $M$ is tangent to $M^g$, and the action of $G$ preserves the complex structure $J$ of $H$. 
Thus, $N^g$  is naturally a complex vector bundle on $M^g$.

We recall the definition of the even and odd equivariant Chern characters. 
For a complex $G$-equivariant vector bundle $E$ on $M$ the equivariant Chern character is as follows. 
Choose a $G$-equivariant connection on $E$, and let $\Omega$ be its curvature. 
Then for each $g\in G$, 
$\ch_g(E)$ is a cohomology class on $M^g$ given by the differential form
\[ \ch_g(E) = \tr \left(g \exp{\left(-\frac{\Omega}{2 \pi i}\right)}\right)\]
With  $f:M\to \mathrm{GL}(r,\CC)$ as above, $\ch_g(f)$ is the cohomology class represented by the differential form
\[\ch_g(f) = -\sum \limits_{k= 0}^n   \frac{1}{(2 \pi i)^{k+1}} \frac{k!}{(2k+1)!} \tr g  (f^{-1}  df)^{2k+1}.
\]
We can now state and prove the equivariant version of Boutet de Monvel's index theorem. Recall from the introduction:
\BDM*
Here:
\begin{itemize}
\item $\chi_g(\pi)=\mathrm{Tr}(\pi(g))$ is the character of a representation $\pi\in R(G)$ evaluated at $g\in G$.
\item $M^g$ is the fixed set $\{p\in M\mid g.p=p\}$. $M^g$ is a contact manifold by restriction of $\theta$.
\item $\Td(M^g)$ is the Todd class of the contact manifold $M^g$, viewed as a stably almost complex manifold. 
\item $N^g$ is the normal bundle of $M^g$ in $M$, viewed as a  complex vector bundle, and $(N^g)^*$ is its dual.
\item $\ch_g$ is the equivariant Chern character.
\item $\ch_g(\lambda_{-1}(N^g)^*)$ is the  alternating sum $\sum (-1)^k \ch_g(\Lambda^k (N^g)^*) $.
\end{itemize}

\begin{proof}
The $G$-equivariant Fredholm operator $T_f$ determines an element in equivariant $K$-homology
\[[T_f]\in K_0^G(M)\]
The  function $f:M\to \mathrm{GL}(k,\CC)$ determines an element in equivariant $K$-theory,
\[[f]\in K^1_G(M)\]
Then $[T_f]=[f]\cap [S]$ is the cap product of $[f]\in K^1_G(M)$ and $[S]\in K_1^G(M)$ (see \cite{BDT89}).
Theorem \ref{thm:main} implies that $[T_f]$ is equal in $K$-homology to $[f]\cap [D]$. The latter can be represented by the ``Toeplitz operator''  $D_f'$ obtained by replacing the Szeg\"o projection $S$ by the positive spectral projection $P_+$ of $D$,
\[D_f' = (P_+\otimes I)M_f(P_+\otimes I)\]
The operator $D_f'$ acts on the range of $P_+\otimes I$.
We may replace $D_f'$ by the operator $D_f$ that acts on $L^2(M,\CC^r)$,
\[ D_f = (P_+\otimes I)M_f(P_+\otimes I) + (I-P_+)\otimes I\]
$D_f$ has the same index as $D'_f$. We get 
\[ \Ind_G T_f=\Ind_G D_f\]
But $D_f$ is an elliptic pseudodifferential operator of order zero,
and the  equivariant index formula \cite{AS3} of Atiyah-Singer  applies.
The derivation of formula (\ref{eq:BdM equivariant intro}) is essentially the same as the derivation of the equivariant Riemann-Roch formula (4.4) in \cite{AS3}.
\end{proof}

\appendix

\section{Boundary Maps in $K$-theory}

If $A$ is a unital $C^*$-algebra with closed two-sided ideal $J\subset A$,
then the boundary map  in  $K$-theory
\[ \partial : K_1(A/J)\to K_0(J)\]
is as follows.
An element $[u]\in K_1(A/J)$ 
is represented by a unitary matrix $u\in \mathrm{GL}(k, A/J)$.
Choose a lift $x\in M_k(A)$ of $u$.
If $[u]\ne 0$ then $x$ will not be invertible,
but $x$ has a ``parametrix''. Choose any $y\in M_k(A)$ such that  
\[r=1-xy\in M_k(J)\qquad s=1-yx\in M_k(J)\]
Then
\begin{equation}\label{eq:K boundary}
\partial([x]) = e- \begin{bmatrix}0&0\\0&1\end{bmatrix}\in K_0(J)\qquad 
e=\begin{bmatrix}
    s^2 & yr(1+r)\\
    xs & 1-r^2
    \end{bmatrix}
\end{equation} 
To verify that $e^2=e$  use $sy=yr$ and $rx=xs$.

The other boundary map is
\begin{equation}\label{eq:K boundary 2}
    \partial:K^0(A/J)\to K^1(J)\qquad \partial([e])=[\exp(2\pi i g)] 
\end{equation} 
Here $e\in M_k(A/J)$ is a projection that represents a class $[e]\in K_0(A/J)$,
and $g\in M_k(A)$ is any self-adjoint lift of $e$.
Note that $\exp(2\pi i g)$ is a unitary in $M_k(A)$
whose image in $M_k(A/J)$ is $\exp(2\pi i e)=1_k$.

We have the isomorphism
\[ KK_1(\CC,J)\cong K_1(J)\]
An element in $KK_1(\CC,J)$ is represented by an adjointable operator
$F\in \mcL(\mcE)$
that acts on a right Hilbert $J$ module $\mcE$, such that
\[ F^*=F\qquad F^2-1\in \KK(\mcE)\]
To obtain an element in $K_1(J)$ take the unitary
\[ [u]\in K_1(\mcK(\mcE))\qquad u=-\exp(\pi i F)\in \KK(\mcE)^+\]
and observe that $\KK(\mcE)$ is Morita equivalent to $J$.

With this in mind, an alternative description of the boundary map (\ref{eq:K boundary 2}) is
\begin{equation}\label{eq:K boundary 3}
    \partial : K_0(A/J)\to KK_1(\CC,J)\qquad \partial([e])=[F]
\end{equation} 
where $F\in M_k(A)$ is a self-adjoint lift of $2e-1\in M_k(A/J)$.
Then
\[ F^*=F\qquad F^2-1\in M_k(J)\]
We have 
$F\in \mcL(J^k)$
by letting $M_k(A)$ act from the left on the right Hilbert $J$ module $J^k$.
Note that $\KK(J^k)=M_k(J)$.
If we let $F=2g-1$, with $g$ as in  (\ref{eq:K boundary 2}), then
\[-\exp(\pi i F) = \exp(2 \pi i g)\]

\section{Deformations in $K$-theory}

Let $\{A_\lambda\}_{\lambda\in [0,1]}$ be a continuous field of $C^*$-algebras that is trivial away from $0$, i.e. the field is trivial when restricted to $(0,1]$.
Such a field determines a map in $K$-theory
\[ K_j(A_0)\to K_j(A_1)\]
For $K_0$ this map is as follows.
Let $A^+_\lambda$ be the $C^*$-algebra $A_\lambda$ with unit adjoined.
A projection $e_0\in M_n(A_0^+)$  extends to a continuous section $\lambda\mapsto e_\lambda$, where each $e_\lambda\in M_n(A_\lambda^+)$ is a projection.
Then $e_0\mapsto e_1$ gives a map $K_0(A_0^+)\to K_0(A_1^+)$.
This map restricts to $K_0(A_0)\to K_0(A_1)$,
where $K_0(A_\lambda)$ is the kernel of $K_0(A_\lambda^+)\to K_0(\CC)$.
For $K_1$, identify $K_1(A)=K_0(C_0(\RR)\otimes A)$.
(See Theorem 3.1 in \cite{ENN93} for details.)

For a $C^*$-algebra $B$, let $\mathfrak{A} B$ denote the asymptotic algebra,
\[\mathfrak{A}B=C_b((0,1],B)/C_0((0,1],B)\]
An {\em asymptotic morphism} from a $C^*$-algebra $A$ to $B$ is a $\ast$-homomorphism $A\to \mfA B$ (see \cite{GHT00}).
An asymptotic morphism determines a map in $K$-theory,
\[ K_j(A)\to K_j(B)\]
This map is the composition of three maps,
\[ K_j(A)\mapsto K_j(\mfA B)\cong K_j(C_b((0,1],B))\to K_j(B)\]
The first map is determined by the $\ast$-homomorphism $A\to \mfA B$.
Since $C_0((0,1],B)$ is contractible, we have an isomorphism in $K$-theory
\[  K_j(\mfA B)\cong K_j(C_b((0,1],B))\]
Finally, evaluation at $\lambda=1$ gives a $\ast$-homomorphism $C_b((0,1], B)\to B$.

Let $\{A_\lambda\}_{\lambda\in [0,1]}$ be a continuous field of $C^*$-algebras that is trivial away from $0$.
Given $a_0\in A_0$, extend it to a continuous section $a_\lambda\in A_\lambda, \lambda\in [0,1]$.
Then $\lambda\mapsto a_\lambda$ defines an element in $C_b((0,1],B)$ which is uniquely determined by $a_0$ up to equivalence mod $C_0((0,1],B)$.
Thus, the continuous field determines an asymptotic morphism,
\[ A_0\mapsto \mfA A_1\]
The map in $K$-theory determined by the continuous field $\{A_\lambda\}_{\lambda\in [0,1]}$ is the same as the map in $K$-theory determined by the asymptotic morphism $A_0\to \mfA A_1$.
We use continuous fields throughout to construct our deformations, but we need asymptotic morphisms for various commutative diagrams.

Consider a commutative diagram
\begin{equation}\label{eq:asymptotic morphism of ses}
\xymatrix{
0\ar[r]
&A\ar[r]\ar@{.>}[d]
&B\ar[r]\ar@{.>}[d]
&C\ar[r]\ar@{.>}[d]
&0\\
0\ar[r]
&D\ar[r]
&E\ar[r]
&F\ar[r]
&0
}
\end{equation}
$A, B, C, D, E, F$ are $C^*$-algebras.
Horizontal arrows are $\ast$-homomorphisms, and both rows are short exact sequences. The three vertical arrows are asymptotic morphisms.
The diagram (\ref{eq:asymptotic morphism of ses}) commutes asymptotically.
Equivalently, the following diagram of $\ast$-homomorphisms commutes,
\[
\xymatrix{
A\ar[r]\ar[d]
&B\ar[r]\ar[d]
&C\ar[d]
\\
\mfA D\ar[r]
&\mfA  E\ar[r]
&\mfA  F
}
\]
\begin{lemma}\label{lem:naturality of del for asymptotic morphism}
Given a diagram (\ref{eq:asymptotic morphism of ses}) that commutes asymptotically,
we obtain a commutative diagram in $K$-theory,
\[ 
\xymatrix{
K_1(C)\ar[r]^-{\partial}\ar[d]
&K_0(A)\ar[d]\\
K_1(F)\ar[r]^-{\partial}
&K_0(D)
}
\]
where $\partial$ denotes the boundary map in $K$-theory. The same holds for the boundary map from $K_0$ to $K_1$.
\end{lemma}
\begin{proof}
By nuclearity of commutative $C^*$-algebras, the functors $C_b((0,1],\,-\,)=C_b((0,1])\otimes$ and $C_0((0,1],\,-\,)=C_0((0,1])\otimes$ are exact, 
and therefore so is the functor $\mfA$.
Diagram (\ref{eq:asymptotic morphism of ses}) can be expanded to a commutative diagram with four short exact sequences, where all arrows are $\ast$-homomorphisms,
\[ 
\xymatrix{
0\ar[r]
&A\ar[r]\ar[d]
&B\ar[r]\ar[d]
&C\ar[r]\ar[d]
&0\\
0\ar[r]
&\mfA D\ar[r]
&\mfA E\ar[r]
&\mfA F\ar[r]
&0\\
0\ar[r]
&C_b((0,1],D)\ar[r]\ar[u]\ar[d]
&C_b((0,1],E)\ar[r]\ar[u]\ar[d]
&C_b((0,1],F)\ar[r]\ar[u]\ar[d]
&0\\
0\ar[r]
&D\ar[r]
&E\ar[r]
&F\ar[r]
&0
}
\]
Naturality of the $K$-theory boundary map $\partial$ then gives a commutative diagram,
\[ 
\xymatrix{
K_1(C)\ar[r]^-{\partial}\ar[d]
&K_0(A)\ar[d]
\\
K_1(\mfA F)\ar[r]^-{\partial}
&K_0(\mfA D)
\\
K_1(C_b((0,1],F))\ar[r]^-{\partial}\ar[u]^{\cong}\ar[d]
&K_0(C_b((0,1],D))\ar[u]^{\cong}\ar[d]
\\
K_1(F)\ar[r]^-{\partial}
&K_0(D)
}
\]
The composition of the vertical maps is precisely the map in $K$-theory determined by the deformations $C\rightsquigarrow F$ and $A\rightsquigarrow D$.

\end{proof}

Lemma \ref{lem:naturality of del for asymptotic morphism}  holds equivariantly. See \cite{GHT00} for details on equivariant asymptotic morphisms.

\section{An Explicit Deformation of the Bott Generator in the Weyl algebra}\label{sec:explicit family of idempotents}

In this appendix we construct an explicit smooth family of idempotents connecting the Bott generator to the vacuum projection.  
We work  directly in the Weyl algebra, using  Mehler's formula (see e.g. \cite{Ta86}), and avoid the use of functional calculus.

Our construction is reminiscent of that in \cite{ENN93}, where a similar deformation is used to prove Bott periodicity. However, the family constructed there does not satisfy certain additional properties required for our purposes. We refine that construction by producing a family  that is invariant under the action of the unitary group on $V^{1,0}$, and whose entries are rapidly decaying functions (up to constants). 

These properties allow us to use this family of idempotents to construct an explicit deformation from the principal symbol of the Dirac operator to the principal Heisenberg symbol of the Szeg\"o projection. In particular, this yields a direct proof of Proposition \ref{prop:Dirac to Szego in even K theory}.

\vskip 12pt

For $\tau>0$, define 
\[K_\lambda(\tau)\in \mcS_\lambda=\mcS(V^*,\#_\lambda)\]
by Mehler's formula (expressed in the Weyl algebra),
\[
K_\lambda(\tau, v) = 
\begin{cases}
e^{-\tau\|v\|^2} &\text{  if }\lambda = 0\\
\frac{ 1 }{(\cosh \lambda \tau)^n}e^{- \frac{\tanh(\lambda \tau)}{\lambda} \|v\|^2}
&\text{  if }\lambda \ne 0 
\end{cases}
\]
Informally, $K_\lambda(\tau)$ is the $\#_\lambda$ analog of $e^{-\tau Q_\lambda}$, where $Q_\lambda(v)=\|v\|^2$ as in Example \ref{ex:harmonic osc}.

\begin{lemma}\label{Mehler}  For $\tau_1, \tau_2 \ge 0$, $K_\lambda(\tau)$ satisfies the semigroup equation
\[
K_\lambda(\tau_1)  K_\lambda(\tau_2) = K_\lambda(\tau_1+\tau_2) 
\]

\[
\frac{\partial K_\lambda(\tau)}{\partial \tau} = - Q_\lambda   K_\lambda(\tau) = -K_\lambda(\tau)   Q_\lambda
\]
 
\end{lemma} 
\begin{proof}
The first equation follows  from Example \ref{composition example} and the composition laws of hyperbolic functions.
The second is verified by a direct calculation, cf. \cite{Ta86} p.75.

\end{proof}

As in section \ref{sec:Bott quantization},
let  $\Lambda=\Lambda^\bullet V^{1,0}$ and let $A_\lambda\in \Wp_\lambda \otimes \End(\Lambda)$ be
\[
A_\lambda:V^*\to \End(\Lambda) \qquad A_\lambda(z)= c(z)=\varepsilon_z+\iota_z
\]
Lemma \ref{lem:Asq} suggest that, again informally,  we can factor $e^{-\tau A_\lambda^2}$ as
\[ e^{-\tau A_\lambda^2}\sim K_\lambda(\tau) e^{-\lambda \tau N}\in \mcS_\lambda\otimes \End(\Lambda)\]

\begin{lemma} \label{lem:commutativity}
In $\mcW_\lambda\otimes\End(\Lambda)$, $A_\lambda$ commutes with $K_\lambda(\tau) e^{-\lambda \tau N}$.
\end{lemma}
\begin{proof}
 $ K_\lambda(\tau)$ is a real analytic function of $\tau>0$, and hence so is the commutator $[A_\lambda, K_\lambda(\tau) e^{-\lambda \tau N}]$. Therefore to check that it equals $0$ it is sufficient to check that all its derivatives in $\tau$ vanish at $\tau=0$. For $\tau>0$
 \begin{multline*}
 \frac{\partial^k}{\partial \tau^k } [A_\lambda, K_\lambda(\tau) e^{-\lambda \tau N}]=
(-1)^k [A_\lambda, (Q_\lambda +\lambda N)^{k} K_\lambda(\tau) e^{-\lambda \tau N}] =\\
(-1)^k A_\lambda^{2k} [A_\lambda, K_\lambda(\tau) e^{-\lambda \tau N}] 
 \end{multline*}
 When $\tau\downarrow 0 $, $K_\lambda(\tau) \to 1$ with all $v$ derivatives.
 Hence $[A, K_\lambda(\tau) e^{-\lambda t N}] \to 0$ with all derivatives as $\tau\downarrow 0 $.
 And hence also  $A_\lambda^{2k} [A_\lambda, K_\lambda(\tau) e^{-\lambda \tau N}]\to 0$. 
 
\end{proof}

 We now construct the idempotents that will interpolate between the Bott generator and the vacuum projection. 
With $\lambda\in [0,1]$ and $\tau\ge 1$ define
\[R_\lambda(\tau):=K_\lambda(\tau)e^{-\lambda \tau N} \in \mcS_\lambda \otimes \End(\Lambda)\] 
and
\[
B_\lambda(\tau):= A_\lambda  \int_\tau^{2\tau} K_\lambda(t) e^{-\lambda t N} dt \in \mcS_\lambda \otimes \End(\Lambda)
\]
The element $R_\lambda(\tau)$ is even with respect to the grading of $\Lambda$, while $B_\lambda(\tau)$ is odd.

\begin{lemma} For all $\tau\ge 1$,
\[
    A_\lambda B_\lambda(\tau)=B_\lambda(\tau) A_\lambda = R_\lambda(\tau)-R_\lambda(\tau)^2
     \]
\end{lemma}
\begin{proof}
By Lemma \ref{lem:commutativity}, $A_\lambda$ commutes with $B_\lambda(\tau)$. Moreover,
\begin{align*}
A_\lambda B_\lambda(\tau) 
&= \int_\tau^{2\tau} (Q_\lambda+\lambda N) K_\lambda(t) e^{-\lambda t N} dt\\
&=
 -\int_\tau^{2\tau}   \frac{\partial}{\partial t} \left( K_\lambda(t) e^{-\lambda t N} \right) dt \\
 &= K_\lambda(\tau) e^{-\lambda \tau N}-K_\lambda(2\tau ) e^{-2  \lambda \tau N}
\end{align*}
By Lemma \ref{Mehler},
\[ K_\lambda(2\tau) e^{-2 \lambda \tau N} = R_\lambda(\tau)^2\]
This proves the result.

\end{proof}

Now write these operators in block form with respect to the grading $\Lambda=\Lambda^+\oplus\Lambda^-$,
\[A_\lambda= \begin{bmatrix} 0& a^*_\lambda\\
a_\lambda &0
\end{bmatrix}
\quad R_\lambda(\tau)= \begin{bmatrix} r_\lambda^+(\tau)& 0\\
0 &r_\lambda^-(\tau)
\end{bmatrix}
\quad B_\lambda(\tau)= \begin{bmatrix}  0 & b_\lambda^*(\tau)\\
b_\lambda(\tau) &0
\end{bmatrix}
\]
The idempotent in the following lemma is obtained by the standard procedure from relative $K$-theory, using formula (\ref{eq:K boundary}).
Heuristically, with parametrix $Y_\lambda$ as in (\ref{informal Y}) we have $R_\lambda=1-Y_\lambda A_\lambda=1-A_\lambda Y_\lambda$ and
 $B_\lambda=Y_\lambda R_\lambda$.

\begin{lemma}
 Let 
 \begin{equation*}\label{eq:e_lambda}
     e_\lambda(\tau) := 
 \begin{bmatrix}  (r^+)^2 & b^*  (1+r^-)\\
a r^+  & 1-(r^-)^2
\end{bmatrix}
 \end{equation*}
 where $a=a_\lambda$, $r^\pm=r^\pm_\lambda(\tau)$ and $b^*=b^*_\lambda(\tau)$.
Then $e_\lambda(\tau)$ is an idempotent in $\mcS^+_\lambda \otimes \End(\Lambda)$, where $\mcS^+_\lambda=\mcS_\lambda + \CC\cdot 1$.
For fixed $\tau\ge 1$, the family $e_\lambda(\tau)$ depends smoothly on $\lambda$. 
\end{lemma}
\begin{proof}
From the identities 
\[ AR=RA\qquad BR=RB\qquad AB=BA=R-R^2\]
we obtain, in block form,
\[ ar^+=r^-a\qquad b^*r^-=r^+b^*\]
and
\[ab^*= r^-- (r^-)^2\qquad b^*a= r^+ - (r^+)^2\]
Using these relations, one checks that $e_\lambda(\tau)^2=e_\lambda(\tau)$. 
Smooth dependence on $\lambda$ follows from smooth dependence of $K_\lambda(\tau)$, hence of $R_\lambda(\tau)$ and $B_\lambda(\tau)$, on $\lambda$.

\end{proof}

\begin{lemma}
With $\lambda=0$, $\tau=1$, the formal difference of idempotents 
\[e_0(1)-\begin{bmatrix}0&0\\0&1\end{bmatrix}\in K_0^U(C_0(V^*))\]
is the Bott generator.
\end{lemma}
\begin{proof}
If $\lambda=0$, equation (\ref{informal Y}) defines a  function $Y_0\in \mcS(V^*)\otimes \End(\Lambda)$,
\[ Y_0(z) = 
\begin{cases}
    (1-e^{-\|z\|^2})c(z)^{-1}& \text{ if } z\ne 0\\
    0&\text{ if } z=0
\end{cases}
\]
Then $e_0(1)$ is homotopic to the idempotent $e_\beta$ of equation (\ref{eq:Bott3}).
The two idempotents are obtained by a different choice of parametrix for $c(z)$.

\end{proof}

If we fix $\tau=1$ we obtain a family of idempotents $e_\lambda(1)\in \mcS^+_\lambda \otimes \End(\Lambda)$ with $\lambda\in [0,1]$. 
Next we fix $\lambda=1$ and let $\tau\to\infty$.

\begin{proposition}
Let $p_0\in \End(\Lambda)$ be the rank one projection $\Lambda^\bullet V^{1,0} \to \Lambda^0 V^{1,0}$ and let $s_1\in \mcS_1$ be the vacuum projection as in Example \ref{composition example}. 
Fix $\lambda=1$. Then as $\tau \to \infty$,
    \[
    e_1(\tau) \to  \begin{bmatrix}  s_1\otimes p_0  & 0\\
0 & 1
\end{bmatrix}
    \]
converges in the Schwartz space topology for $\mcS^+_1\otimes \End(\Lambda)$.    
\end{proposition}
\begin{proof}
We analyze the entries of $e_1(\tau)$ as $\tau \to \infty$. 
On $\Lambda^k V^{1,0}$, the number operator $N$ acts by multiplication by $2k-n$, and therefore
\[
R_1(\tau) = \frac{ e^{- \tanh(\tau ) \|v\|^2} }{(\cosh \tau )^n} e^{(-2k+n)\tau}
\]
Using
\[ \frac{1}{(\cosh{\tau})^n} e^{n\tau} = \left(\frac{2}{1+e^{-2\tau}}\right)^n\]
we obtain
\[R_1(\tau)|_{\Lambda^kV^{1,0}} =\frac{ 2^n e^{- \tanh(\tau ) \|v\|^2} }{(1+e^{-2\tau} )^n} e^{- 2k \tau}\]
For $k=0$, this converges in the Schwartz topology to the vacuum projection
\[ R_1(\tau)|_{\Lambda^0V^{1,0}} \to 2^ne^{-\|v\|^2} = s_1\]
For $k>0$, the extra factor $e^{-2k\tau}$ implies that 
\[R_1(\tau)|_{\Lambda^kV^{1,0}}\to 0\]
in the Schwartz topology as $\tau\to\infty$. Hence we obtain for the diagonal entries,
\[r^+_1 (\tau) \to s_1\otimes p_0\qquad r^-_1(\tau) \to 0\]
in $\mcS_1\otimes \End(\Lambda)$.

Next consider
\[\int_\tau^{2 \tau}K_1(t)e^{-tN}dt\]
On $\Lambda^k V^{1,0}$ this acts by multiplication with 
\[
\int_\tau^{2 \tau} \frac{2^n e^{- \tanh(t) \|v\|^2} }{(1+e^{-2t})^n} e^{-2kt}dt 
\]
For $k>0$, this converges to $0$ in the Schwartz topology as $\tau\to\infty$, since every Schwartz seminorm is bounded by a constant multiple of $\int_\tau^{2\tau} e^{-2kt}dt$, uniformly in $v$.
It follows that the off-diagonal block $b_1^*(\tau)\to 0$.

It remains to control the lower-left entry $a r^+_1(\tau)$. 
Since $ar^+=r^-a$, and $r^-_1(\tau)\to 0$, we obtain $ar^+_1(\tau)\to 0$.

Now
\[ e_1(\tau) 
 =
 \begin{bmatrix}  r^+_1(\tau)^2 & b^*_1(\tau)  (1+r^-_1(\tau))\\
a r^+_1(\tau)  & 1-(r^-_1(\tau))^2
\end{bmatrix}
\]
Since each matrix entry converges in the Schwartz topology, we conclude that
\[e_1(\tau) 
\to  
\begin{bmatrix}  s_1\otimes p_0 & 0\\
0  & 1
\end{bmatrix}\]

\end{proof}

Concatenating the path $\{e_\lambda(1)\}_{\lambda\in [0,1]}$ with the path $\{e_1(\tau)\}_{\tau\in [1,\infty]}$, we obtain a continuous family of idempotents in $\mcS^+_\lambda\otimes \End(\Lambda)$.
This provides an explicit deformation from the Bott generator to the vacuum projection that is
invariant under the action of the unitary group on $V^{1,0}$.
\vfill

\section{Notation}

\begin{tabular}{ll}

$M$ & compact $(2n+1)$-dimensional contact manifold \\[2pt]

$\theta$ & contact form on $M$ \\[2pt]

$H=\ker\theta$ & contact hyperplane bundle \\[2pt]

$\omega=d\theta|_{H}$ & symplectic form on $H$ \\[2pt]

$J$ & compatible complex structure on $(H,\omega)$ or $(V,\omega)$ \\[2pt]

$H^{1,0}$ & complex bundle determined by $J$ \\[2pt]

$V$ & $2n$-dimensional symplectic vector space \\[2pt]

$V^{1,0}$ & complex vector space determined by $(V,\omega,J)$ \\[2pt]

$\Heis=V\times\RR$ & Heisenberg group with law 
$(v,t)(w,s)=(v+w,t+s-\tfrac12\omega(v,w))$ \\[2pt]

$\mcS(V)$ & Schwartz class functions on $V$ \\[2pt]

$\ast_\lambda$ & twisted convolution on $\mcS(V)$ \\[2pt]

$\#_\lambda$ & Moyal product on $\mcS(V^*)$ \\[2pt]

$\Wp_\lambda$ & polynomial Weyl algebra \\[2pt]

$\mcW_\lambda$ & Weyl symbol algebra \\[2pt]

$W_\lambda^*$ & $C^*$-completion of $\mcW_\lambda^0$ \\[2pt]

$\sigma^W$ & principal Weyl symbol \\[2pt]

$\Lambda=\Lambda^\bullet V^{1,0}$ & spinor space, 
$\Lambda=\Lambda^+\oplus\Lambda^-$ \\[2pt]

$\varepsilon_z,\iota_z$ & exterior and interior multiplication on $\Lambda$ \\[2pt]

$c(z)=\varepsilon_z+\iota_z$ & Clifford multiplication \\[2pt]

$\slashed{S}=\Lambda^\bullet H^{1,0}$ & Spinor bundle on $M$ \\[2pt]

$D$ & Spin$^c$ Dirac operator on $M$ \\[2pt]

$P_+$ & positive spectral projection of $D$ \\[2pt]

$S$ & Szeg\"{o} projection \\[2pt]

$\sigma(D)$ & classical principal symbol of $D$ \\[2pt]

$\sigma_H(\mcL)$ & Heisenberg principal symbol of $\mcL$ \\[2pt]

$\SymH^0$ & order-zero Heisenberg symbol algebra \\[2pt]

$\SymH^*$ & $C^*$-completion of $\SymH^0$ \\[2pt]

$T_HM$ & Heisenberg tangent groupoid \\[2pt]

$C^*(T_HM)$ & groupoid $C^*$-algebra \\[2pt]

$K_j(\cdot)$ & $K$-theory \\[2pt]

$K_j^G(\cdot)$ & equivariant $K$-theory \\[2pt]

$KK(\cdot,\cdot)$ & Kasparov $KK$-theory \\

\end{tabular}
\vfill

\bibliographystyle{amsplain}
\bibliography{MyBibfile}

\end{document}